\documentclass[12pt,a4paper]{amsart}
\usepackage{amssymb}
\usepackage[T2A]{fontenc}
\usepackage[english]{babel}
\usepackage{srcltx}
\usepackage{upgreek}

\def\w{{\omega }}
\def\T{{\mathbb T}}
\def\d{{\delta}}

\def\a{{\alpha}}

\advance\textwidth20mm \advance\hoffset-5mm
\advance\textheight15mm \advance\voffset-15mm
\theoremstyle{plain}
\newtheorem{theorem}{Theorem}[section]
\newtheorem{corollary}[theorem]{Corollary}
\newtheorem{lemma}[theorem]{Lemma}

\newtheorem{proposition}{Proposition}[section]
\theoremstyle{remark}
\newtheorem{remark}{Remark}[section]

\newtheorem*{example}{Example}
\numberwithin{equation}{section}

\newcommand\R{\mathbb{R}}
\newcommand\N{\mathbb{N}}
\newcommand\Z{\mathbb{Z}}

\newcommand\<{\langle}
\renewcommand\>{\rangle}

\newcommand\Cdot{{\mskip2mu{\cdot}\mskip2mu}}
\newcommand\Ast[1]{\ast_{\scriptscriptstyle\mskip-2mu #1}}

\begin{document}

\title[Fractional modulus and $K$-functional in $L^p$ with Dunkl weight]
{Fractional smoothness in $L^p$ with Dunkl weight and its applications}

\author{D.~V.~Gorbachev}
\address{D.~Gorbachev, Tula State University,
Department of Applied Mathematics and Computer Science,
300012 Tula, Russia}
\email{dvgmail@mail.ru}

\author{V.~I.~Ivanov}
\address{V.~Ivanov, Tula State University,
Department of Applied Mathematics and Computer Science,
300012 Tula, Russia}
\email{ivaleryi@mail.ru}

\date{\today}

\thanks{This work is supported by the Russian Science Foundation under grant 18-11-00199.}

\begin{abstract}
We define fractional power of the Dunkl Laplacian, fractional modulus of smoothness and fractional $K$-functional
in $L^p$-space with the Dunkl weight. As application, we prove direct and inverse theorems of approximation theory,
and some inequalities for entire functions of spherical exponential type in fractional settings.
\end{abstract}

\maketitle
\bigskip
\tableofcontents

\section{Introduction}
During { the last three decades}, many important elements of harmonic analysis with Dunkl
weight on $\R^d$ and ${\mathbb{S}}^{d-1}$ were proved; see, e.g., the papers by
C.F.~Dunkl \cite{Dun89,Dun91,Dun92}, M.~R\"{o}sler
\cite{Ros98,Ros99,Ros03,Ros02}, M.F.E.~de~Jeu \cite{Jeu93,Jeu06},
K.~Trim\`{e}che \cite{Tri01,Tri02}, Y.~Xu \cite{Xu97,Xu00}, and the recent
works \cite{anker, Dai13, Dai15, iv1, iv2}.

The classical translation operator $f\mapsto f(\Cdot+y)$ plays an important
role both in approximation theory and harmonic analysis, in particular, to
introduce several smoothness characteristics of $f$. In Dunkl harmonic analysis
its analogue is the generalized translation operator $\tau^y$ defined by
M.~R\"{o}sler \cite{Ros98}. Unfortunately, the $L^p$-boundedness of $\tau^y$ is
not established in general.

To overcome this difficulty, the spherical mean value of the
translation operator $\tau^y$ was introduced in \cite{MejTri01} and it was
studied in \cite{Ros03}, where, in particular, its positivity was shown.
In \cite{iv3} we proved that this operator
is a positive $L^p$-bounded operator $T^t$, which may be considered as a generalized
translation operator. It is worth mentioning that this operator can be applied
to problems where it is essential to deal with radial multipliers.

One application of this operator is basic inequalities of approximation theory in the
weighted $L^p$ spaces. With the help of the operator $T^t$ and radial multipliers
from $C^{\infty}(\mathbb{R}^d)$ we defined in \cite{iv3} integer power of Dunkl Laplacian,
moduli of smoothness, $K$-functional and proved the direct and inverse approximation theorems,
equivalence between moduli of smoothness and $K$-functional,
weighted analogues of Nikol'skii, Bernstein, and Boas inequalities for
entire functions of spherical exponential type.

In this paper we solve the same problems in fractional case.
We define fractional power of Dunkl Laplacian, fractional modulus of smoothness,
fractional $K$-functional and prove the direct
and inverse approximation theorems, equivalence between modulus of smoothness
and $K$-functional, some weighted inequalities for entire functions of spherical
exponential type in fractional settings. The main difficulty is that the multipliers
that determine smoothness characteristics have a singularity at zero.
To overcome this difficulty, instead of the Schwartz space and tempered distributions
we use the weighted analogue of the Lizorkin space (see, \cite{Liz63,iv4,SamKilMar93}),
the space of Dunkl transforms of its functions, and distributions on these spaces.

Let us now discuss some known results for fractional moduli of smoothness and $K$-functionals
in the non-weighted case.
The modulus of smoothness of order $\alpha$ of a function $f \in L_p(X)$, $X=\R^d,\T^d$, is given by
\begin{equation*}
\omega_{\a} (\delta,f)_{L_p(X)}
 :=\sup_{|h|\le \delta }
 \bigg\|
\sum\limits_{\nu=0}^\infty(-1)^{\nu}
\binom{\a}{\nu} f\,\big(\cdot+(\a-\nu) h\big) \bigg\|_{L_p(X)}, 
\end{equation*}
$\binom{\a}{\nu}=\frac{\a (\a-1)\dots (\a-\nu+1)}{\nu!}$,
$\binom{\a}{0}=1$; for the main properties see~\cite{butzer, jat}.
For definiteness, consider $f \in L_p(\T^d)$.
It is known that  (for see, e.g.,~\cite{KT}), for any $1<p<\infty$, and $\a>0$,
\begin{equation}\label{RTD}
K(f,t^{\a};L_p(\T^d),W_p^\a(\T^d))\asymp R(f,t^{\a};L_p(\T^d),\mathcal{T}_{[1/t]})\asymp \omega_\a(t,f)_{L_p(\T^d)},
\end{equation}
where
the $K$-functional of $f$ is given by
$$
  K(f,t^{\a},L_p(\T^d);W_p^{\a}(\T^d)):= \inf \left\{\|f-g\|_{L_p(\T^d)} +
   t^{\a}\|g \|_{\dot W_p^{\a}(\T^d)}\colon g\in W_p^{\a}(\T^d) \right\}
$$
and the realization of the $K$-functional is defined by
$$
R(f,t^{\a};L_p(\T^d),\mathcal{T}_{[1/t]}):= \inf \left\{\|f-T\|_{L_p(\T^d)} +
t^{\a}\|T\|_{\dot W_p^{\a}(\T^d)}\colon T\,\in \mathcal{T}_{[1/t]} \right\}.
$$
Here, $W_p^{\a}(X)$ is the fractional Sobolev space, i.e.,
$$
W_p^{\a}(X)=\{g\in L_p(X)\colon \|g \|_{\dot W_p^{\a}(X)}=\|(-\Delta)^{\a/2} g\|_{L_p(X)}<\infty\}
$$
and $\mathcal{T}_n$ is
the space of all trigonometric polynomials of order at most $n$, i.e.,
\[
\mathcal{T}_n=\mathrm{span}\,\bigl\{e^{i\<k,x\>}\colon\,k\in\mathbb{Z}^d,\,\,|k|_{\infty}=\max_j|k_j|\leq n\bigr\}.
\]
The key result to obtain (\ref{RTD}) is the following
Nikol'skii--Stechkin--Boas-type inequality \cite{KT} on the relationship
between  norms of derivatives and differences of trigonometric polynomials:
$$
\sup_{\xi\in \R^d,\,|\xi|=1}\Bigl\|\Bigl(\frac{\partial}{\partial\xi}\Bigr)^{\a} T_n\Bigr\|_{L_p(\T^d)}
\asymp \d^{-\a}\w_\a (\d,T_n)_{L_p(\T^d)},\quad T_n\in \mathcal{T}_n,
$$
where $0<\d\le \pi/n$, $0<p\le\infty$,  $\a>0$, and
$(\frac{\partial}{\partial\xi})^{\a} T$ is
 the directional derivative of order $\a$, that is, 
$$
\Bigl(\frac{\partial}{\partial\xi}\Bigr)^{\a} T(x)=\sum_{|k|_\infty\le [1/\d]}
{(i\<k,\xi\>)^\a} c_k e^{i\<k,x\>}.
$$

The direct and inverse inequalities are written as follows:
$$E_n(f)_{L_p(\T^d)}\lesssim \omega_\a(n^{-1},f)_{L_p(\T^d)}\lesssim
\sum_{k=1}^nk^{\a-1}E_{k-1}(f)_{L_p(\T^d)},\qquad \a>0,
$$
where
$E_n(f)_{p}$ is the best approximation of $f\in L_p(\T^d)$ by trigonometric polynomials $T\in\mathcal{T}_n$.

Similar results for functions on $\R^d$ can be found in \cite{gogat, jat, wil}.

The paper is organized as follows. In the next section, we give some basic
notation and facts of Dunkl harmonic analysis.
In Section~2 we introduce necessary spaces of distributions and define
fractional power of Dunkl Laplacian, fractional modulus of smoothness
and fractional $K$-functional, associated to the Dunkl weight.
In Section 4, we prove equivalence between them as well as the Jackson
inequality. Section~5 consists of some weighted inequalities for entire
functions of exponential type in fractional settings.
In Section~6, we obtain that modulus of smoothness are equivalent to the
realization of the $K$-functional. We conclude with Section~7, where we prove
the inverse theorems in $L^p$-spaces with the Dunkl weight.

\bigskip
\section{Notation and elements of Dunkl harmonic analysis}
In this section, we recall  the basic notation and results of the Dunkl harmonic analysis, see, e.g., 
\cite{Ros02}.

Throughout the paper, $\<x,y\>$ denotes the standard Euclidean scalar product in
 $d$-dimensional Euclidean space $\mathbb{R}^d$, $d\in \N$,
equipped with a norm $|x|=\sqrt{\<x,x\>}$.
For $r>0$ we write $B_r=\{x\in\R^d\colon |x|\leq r\}$.
Let $\Pi$ be the set of all polynomials of $d$ variables.
For $\alpha=(\alpha_1,\dots,\alpha_d)\in \mathbb{Z}^d_+$, a monomial $x^{\alpha}=\prod_{j=1}^dx^{\alpha_j}$ has degree
$|\alpha|_1=\sum_{j=1}^d\alpha_j$. The degree of polynomial is the greatest degree of its monomials.
$\Pi_m$ denotes the set of all polynomials of degree at most $m\in\mathbb{Z}_+$.

We will assume that $A\lesssim B$ means that $A\leq CB$ with a
constant $C>0$ depending only on nonessential parameters. Asymptotical equality $A\asymp B$ means
that $A\lesssim B$ and $B\lesssim A$. For $\alpha, \beta\in\mathbb{Z}^d_+$
the inequality $\alpha\leq\beta$ means that $\alpha_j\leq\beta_j$, $j=1,\dots,d$.


Define the following function spaces:
\begin{itemize}
  \item $C_b(\mathbb{R}^d)$ the space of bounded continuous functions with the norm
$\|f\|_{\infty}=\sup_{\mathbb{R}^d}|f|$,
  \item $C_0(\mathbb{R}^d)$ the space of continuous functions
which vanish at infinity,
  \item $C^{\infty}(\mathbb{R}^d)$ the space of infinitely
differentiable functions,
  \item $C^{\infty}_{\Pi}(\mathbb{R}^d)$ the space of infinitely
differentiable functions whose derivatives have polynomial growth at infinity,
  \item $C^{\infty}_{\Pi}(\mathbb{R}^d\setminus\{0\})$=$\{|x|^pf(x)\colon\,f\in C^{\infty}_{\Pi}(\mathbb{R}^d),p\in\mathbb{R}\}$,
  \item $\mathcal{S}(\mathbb{R}^d)$ the Schwartz space,
  \item $\mathcal{S}'(\mathbb{R}^d)$
  the space of tempered distributions,
  \item $X(\mathbb{R}_+)$ the space of even functions from $X(\mathbb{R})$, where $X$ is one of the spaces above,
  \item
 $X_\mathrm{rad}(\mathbb{R}^d)$ the subspace of $X(\mathbb{R}^d)$ consisting of radial functions
 $f(x)=f_{0}(|x|)$.

\end{itemize}

Let a finite subset $R\subset \mathbb{R}^{d}\setminus\{0\}$ be a root system,
$R_{+}$ positive subsystem of $R$, $G(R)\subset O(d)$ finite reflection group,
generated by reflections $\{\sigma_{a}\colon a\in R\}$, where $\sigma_{a}$ is a
reflection with respect to hyperplane $\<a,x\>=0$, $k\colon R\to
\mathbb{R}_{+}$ $G$-invariant multiplicity function. Recall that a finite
subset $R\subset \mathbb{R}^{d}\setminus\{0\}$ is called a root system, if
\[
R\cap\R a=\{a, -a\}\quad \text{and}\quad \sigma_{a}R=R\ \text{for all}\ a\in R.
\]

Let
\[
v_{k}(x)=\prod_{a\in R_{+}}|\<a,x\>|^{2k(a)}
\]
be the Dunkl weight,
\[
c_{k}^{-1}=\int_{\mathbb{R}^{d}}e^{-|x|^{2}/2}v_{k}(x)\,dx,\quad d\mu_{k}(x)=c_{k}v_{k}(x)\,dx,
\]
and $L^{p}(\mathbb{R}^{d},d\mu_{k})$, $0<p<\infty$, be the space of
complex-valued Lebesgue measurable functions $f$ for which
\[
\|f\|_{p,d\mu_{k}}=\Bigl(\int_{\mathbb{R}^{d}}|f|^{p}\,d\mu_{k}\Bigr)^{1/p}<\infty.
\]
We also assume that $L^{\infty}\equiv C_b$ and $\|f\|_{\infty,d\mu_{k}}=\|f\|_{\infty}$.

\begin{example} If the root system $R$ is $\{\pm e_1,\dots,\pm e_d\}$, where
$\{e_1,\dots,e_d\}$ is an orthonormal basis of $\mathbb{R}^{d}$, then
$v_{k}(x)=\prod_{j=1}^d|x_j|^{2k_j}$, $k_{j}\ge 0$, $G=\mathbb{Z}_2^d$.
\end{example}

Let
\[
D_{j,k}f(x)=D_{j}f(x)=\frac{\partial f(x)}{\partial x_{j}}+
\sum_{a\in R_{+}}k(a)\<a,e_{j}\>\,\frac{f(x)-f(\sigma_{a}x)}{\<a,x\>},\quad
j=1,\dots,d,
\]
be differential-differences Dunkl operators and $\Delta_k=\sum_{j=1}^dD_j^2$ be the
Dunkl Laplacian. The Dunkl kernel $e_{k}(x, y)=E_{k}(x, iy)$ is a unique
solution of the system
\[
D_{j}f(x)=iy_{j}f(x),\quad j=1,\dots,d,\qquad f(0)=1,
\]
and it plays the role of a generalized exponential function. Its properties are
similar to those of the classical exponential function $e^{i\<x, y\>}$. Several basic properties
follow from an integral representation 
\cite{Ros99}:
\[
e_k(x,y)=\int_{\mathbb{R}^d}e^{i\<\xi,y\>}\,d\mu_x^k(\xi),
\]
where $\mu_x^k$ is a probability Borel measure
and\, $\mathrm{supp}\,\mu_x^k\subset \mathrm{co}{}(\{gx\colon g\in G(R)\})$.
In particular,
\begin{equation*}
|e^{i(x, y)}|\leq 1,\quad e_{k}(x, y)=e_{k}(y, x), \quad e_{k}(-x, y)=\overline{e_{k}(x, y)}.
\end{equation*}





For $f\in L^{1}(\mathbb{R}^{d},d\mu_{k})$, the Dunkl transform is defined by the equality
\[
\mathcal{F}_{k}(f)(y)=\int_{\mathbb{R}^{d}}f(x)\overline{e_{k}(x,
y)}\,d\mu_{k}(x).
\]
For $k\equiv0$, \ $\mathcal{F}_{0}$ is the classical Fourier transform $\mathcal{F}$. We
also note  that $\mathcal{F}_k(e^{-|\,\cdot\,|^2/2})(y)=e^{-|y|^2/2}$ and
$\mathcal{F}_{k}^{-1}(f)(x)=\mathcal{F}_{k}(f)(-x)$.
Let
\begin{equation*}
\mathcal{A}_k=\Big\{f\in L^{1}(\mathbb{R}^{d},d\mu_{k})\cap C_b(\R^d)\colon
\mathcal{F}_{k}(f)\in L^{1}(\mathbb{R}^{d},d\mu_{k})\Big\}.
\end{equation*}
Let us now list several basic of the properties of the Dunkl transform.

\begin{proposition}\label{prop2.1}
\textup{(1)} For $f\in L^{1}(\mathbb{R}^{d},d\mu_{k})$, $\mathcal{F}_{k}(f)\in C_0(\R^d)$.

\smallbreak
\textup{(2)} If $f\in\mathcal{A}_k$, we have the pointwise inversion formula
\[
f(x)=\int_{\mathbb{R}^{d}}\mathcal{F}_{k}(f)(y)e_{k}(x, y)\,d\mu_{k}(y).
\]

\textup{(3)} The Dunkl transform leaves the Schwartz space $\mathcal{S}(\R^d)$ invariant.

\smallbreak
\textup{(4)} The Dunkl transform extends to a unitary operator in $L^{2}(\mathbb{R}^{d},d\mu_{k})$.
\end{proposition}

Let $\lambda\ge -1/2$ and $J_{\lambda}(t)$ be the classical Bessel function of degree $\lambda$ and
\[
j_{\lambda}(t)=2^{\lambda}\Gamma(\lambda+1)t^{-\lambda}J_{\lambda}(t)
\]
be the normalized Bessel function. Set
\[
b_{\lambda}^{-1}=\int_{0}^{\infty}e^{-t^{2}/2}t^{2\lambda+1}\,dt=2^{\lambda}\Gamma(\lambda+1),\qquad
d\nu_{\lambda}(t)=b_{\lambda}t^{2\lambda+1}\,dt,\quad t\in \R_{+}.
\]
The norm in $L^{p}(\R_{+},d\nu_{\lambda})$, $1\le p<\infty$,
is given by
\[
\|f\|_{p,d\nu_{\lambda}}=\biggl(\int_{\R_{+}}|f(t)|^{p}\,d\nu_{\lambda}(t)\biggr)^{1/p}.
\]

The Hankel transform is defined as follows
\[
\mathcal{H}_{\lambda}(f)(r)=\int_{\R_{+}}f(t)j_{\lambda}(rt)\,d\nu_{\lambda}(t),\quad r\in \mathbb{R}_{+}.
\]
It is a unitary operator in $L^{2}(\R_{+},d\nu_{\lambda})$ and
$\mathcal{H}_{\lambda}^{-1}=\mathcal{H}_{\lambda}$ \cite[Chap.~7]{BatErd53}.

Note that if $\lambda=d/2-1$, the Hankel transform is a restriction of the
Fourier transform on radial functions and if
$\lambda=\lambda_k=d/2-1+\sum_{a\in R_+}k(a)$ of the Dunkl transform.
If $f(x)=f_0(|x|)$, then
\begin{equation}\label{restriction}
\mathcal{F}_{k}(f)(y)=\mathcal{H}_{\lambda_k}f_0(|y|).
\end{equation}

Let $\mathbb{S}^{d-1}=\{x'\in\R^d\colon |x'|=1\}$ be the Euclidean sphere and
$d\sigma_k(x')=a_kv_k(x')\,dx'$ be the probability measure on $\mathbb{S}^{d-1}$.
We have
\begin{equation*}
\int_{\mathbb{R}^{d}}f(x)\,d\mu_{k}(x)=
\int_{0}^{\infty}\int_{\mathbb{S}^{d-1}}f(tx')\,d\sigma_k(x')\,d\nu_{\lambda_k}(t).
\end{equation*}

We need the following partial case of the Funk--Hecke formula \cite{Xu00}
\begin{equation}\label{averaging}
\int_{\mathbb{S}^{d-1}}e_{k}(x, ty')\,d\sigma_k(y')=j_{\lambda_k}(t|x|).
\end{equation}

Let $y\in \mathbb{R}^d$ be given. M. R\"{o}sler \cite{Ros98} defined a
generalized translation operator $\tau^y$ in $L^{2}(\mathbb{R}^{d},d\mu_{k})$
by the equation
\[
\mathcal{F}_k(\tau^yf)(z)=e_k(y, z)\mathcal{F}_k(f)(z).
\]
Since $|e_k(y, z)|\leq 1$ then $\|\tau^y\|_{2\to 2}\leq1$.

The operator $\tau^yf$ is not positive in common case (see \cite{Ros95,ThaXu05})
and it remains an open question whether $\tau^yf$ is an $L^p$ bounded operator on $\mathcal{S}(\R^d)$ for $p\neq 2$. It is known only for $G=\Z_2^d$ (\cite{Ros95,ThaXu05}).

Some more we can say about the operator $\tau^y$, considering it on a subspace of radial functions.
\begin{proposition}[\cite{Ros03,ThaXu05,iv3}]\label{prop2.2}
\textup{(1)} If $f\in\mathcal{A}_{k,\mathrm{rad}}$, then pointwise
\begin{equation*}
\tau^yf(x)=\int_{\mathbb{R}^d}e_k(y,z)e_k(x,z)\mathcal{F}_k(f)(z)\,d\mu_k(z).
\end{equation*}

\smallbreak
\textup{(2)} The operator $\tau^yf$ is positive on radial functions. If $f\in C_{b,\mathrm{rad}}(\mathbb{R}^d)$, then
\begin{equation*}
\tau^yf(x)=\int_{\mathbb{R}^d}f(z)\,d\rho_{x,y}^k\in C_b(\mathbb{R}^d\times\mathbb{R}^d),
\end{equation*}
where $\rho_{x,y}^k$ is a radial probability Borel measure and $\mathrm{supp}\,\rho_{x,y}^k\subset B_{|x|+|y|}$. In partial, $\tau^y1=1$.

\smallbreak
\textup{(3)} If $f\in \mathcal{S}_{\mathrm{rad}}(\mathbb{R}^d)$, $1\leq p\leq\infty$, then
$\|\tau^yf\|_{p,d\mu_k}\leq \|f\|_{p,d\mu_k}$
and the operator $\tau^t$ can be extended to $L^{p}_{\mathrm{rad}}(\mathbb{R}^{d},d\mu_{k})$ with
preservation of the norm.
\end{proposition}

Let
\begin{equation}\label{Weighteddimension}
\lambda_k=d/2-1+\sum_{a\in R_+}k(a),\quad d_k=2(\lambda_k+1).
\end{equation}
The number $d_k$ plays the role of the
generalized dimension of the space $(\mathbb{R}^d,d\mu_k)$.
We have $\lambda_k\geq-1/2$ and, moreover, $\lambda_k=-1/2$ only if $d=1$ and
$k\equiv0$. In what follows we assume that $\lambda_k>-1/2$ and $d_k>1$.

Let $t\in \mathbb{R}_+$. In \cite{iv3} we defined new generalized translation operator $T^t$  in $L^{2}(\mathbb{R}^{d},d\mu_{k})$ by relation
\begin{equation}\label{multiplier}
\mathcal{F}_k(T^tf)(y)=j_{\lambda_k}(t|y|)\mathcal{F}_k(f)(y).
\end{equation}
Since $|j_{\lambda_k}(t)|\leq 1$, then $\|T^t\|_{2\to 2}\leq1$.

Let us list several basic of the properties of $T^t$, $t\in\mathbb{R}_+$.

\begin{proposition}[\cite{iv3,Ros03}]\label{prop2.3}
\textup{(1)} If $f\in\mathcal{A}_k$, then pointwise
\[
T^tf(x)=\int_{\mathbb{R}^d}j_{\lambda_k}(t|y|)e_k(x,
y)\mathcal{F}_k(f)(y)\,d\mu_k(y)=\int_{\mathbb{S}^{d-1}}\tau^{ty'}f(x)\,d\sigma_k(y').
\]

\smallbreak
\textup{(2)} The operator $T^t$ is positive. If $f\in C_b(\mathbb{R}^d)$, then
\begin{equation*}
T^tf(x)=\int_{\mathbb{R}^d}f(z)\,d\sigma_{x,t}^k(z)\in C_b(\mathbb{R}_+\times\mathbb{R}^d),
\end{equation*}
where $\sigma_{x,t}^k$ is a probability Borel measure and\, $\mathrm{supp}\,\sigma_{x,t}^k\subset \bigcup\limits_{g\in G}\{z\in\mathbb{R}^d\colon |z-gx|\leq t\}$. In partial, $T^t1=1$.

\smallbreak
\textup{(3)} If $f\in \mathcal{S}(\mathbb{R}^d)$, $1\leq p\leq\infty$, then $\|T^tf\|_{p,d\mu_k}\leq \|f\|_{p,d\mu_k}$  and the operator $T^t$ can be extended to $L^{p}(\mathbb{R}^{d},d\mu_{k})$ with preservation of the norm.

\end{proposition}


Note that for
$k\equiv0$, $T^t$ is the usual spherical mean
\begin{equation*}
T^tf(x)=S^tf(x)=\int_{\mathbb{S}^{d-1}}f(x+ty')\,d\sigma_0(y').
\end{equation*}

Let $g(y)=g_0(|y|)$ be a radial function. S.~Thangavelu and Yu.~Xu~\cite{ThaXu05} defined a convolution
\begin{equation}\label{convolution1}
(f\ast_{k}g)(x)=\int_{\mathbb{R}^d}f(y)\tau^{x}g(-y)\,d\mu_{k}(y).
\end{equation}
\begin{proposition}[\cite{ThaXu05,iv3}]\label{prop2.4}
\textup{(1)} If $f\in\mathcal{A}_k$, $g\in L^{1}_{\mathrm{rad}}(\mathbb{R}^d,d\mu_{k})$, then
\begin{equation*}
(f\Ast{k}g)(x)=\int_{\R^d}\tau^{-y}f(x)g(y)\,d\mu_k(y)\in \mathcal{A}_k,
\end{equation*}
and
\begin{equation*}
\mathcal{F}_k(f\ast_{k}g)(y)=\mathcal{F}_k(f)(y)\mathcal{F}_k(g)(y),\quad y\in\mathbb{R}^d.
\end{equation*}

\smallbreak
\textup{(2)} Let $1\leq p\leq\infty$. If $f\in L^{p}(\mathbb{R}^{d},d\mu_{k})$,
$g\in L^{1}_{\mathrm{rad}}(\mathbb{R}^d,d\mu_{k})$, then
$(f\ast_{k}g)\in L^{p}(\mathbb{R}^{d},d\mu_{k})$, and
\begin{equation}\label{ineq}
\|(f\ast_{k}g)\|_{p,d\mu_k}\leq \|f\|_{p,d\mu_k}\|g\|_{1,d\mu_k}.
\end{equation}
\end{proposition}

\begin{remark}
The inequality \eqref{ineq} was proved in \cite{ThaXu05} under
additional condition of boundedness $g$. This condition can be omitted. Indeed, by
H\"{o}lder's inequality
\begin{align*}
|(f\Ast{k}g)(x)|&=\Bigl|\int_{\mathbb{R}^d}f(y)\tau^{x}g(-y)\,d\mu_{k}(y)\Bigr|\\&\leq
\Bigl(\int_{\mathbb{R}^d}|f(y)|^p|\tau^{x}g(-y)|\,d\mu_{k}(y)\Bigr)^{1/p}
\Bigl(\int_{\mathbb{R}^d}|\tau^{x}g(-y)|\,d\mu_{k}(y)\Bigr)^{1-1/p},
\end{align*}
and by Proposition \ref{prop2.2}
\begin{align*}
\|(f\Ast{\lambda_k}g_0)\|_{p,d\nu_{\lambda_k}}&\leq
\Bigl(\int_{\mathbb{R}^d}\int_{\mathbb{R}^d}|f(y)|^p|\tau^{x}g(-y)|\,d\mu_{k}(y)\,d\mu_k(x)\Bigr)^{1/p}\|g\|_{1,d\mu_{k}}^{1-1/p}\\&
=\Bigl(\int_{\mathbb{R}^d}|f(y)|^p\int_{\mathbb{R}^d}|\tau^{-y}g(x)|\,d\mu_{k}(x)\,d\mu_k(y)\Bigr)^{1/p}\|g\|_{1,d\mu_{k}}^{1-1/p}
\\&\leq\|f\|_{p,d\mu_k}\|g\|_{1,d\mu_{k}}.
\end{align*}
\end{remark}

\bigskip
\section{Best approximation, smoothness characteristics \\ and the $K$-functional}
Let $\mathbb{C}^d$ be the complex Euclidean space of $d$ dimensions.
Let also $z=(z_1,\dots,z_d)\in \mathbb{C}^d$,
$\mathrm{Im}\,z=(\mathrm{Im}\,z_1,\dots,\mathrm{Im}\,z_d)$, and $\sigma>0$.

We define two classes of entire functions: $B_{p, k}^\sigma$ and
$\widetilde{B}_{p, k}^\sigma$. We say that a function $f\in B_{p, k}^\sigma$
if $f\in L^{p}(\mathbb{R}^{d},d\mu_{k})$ is such that its analytic
continuation to $\mathbb{C}^d$ satisfies
\begin{equation*}
|f(z)|\leq c_{\varepsilon}e^{(\sigma+\varepsilon)|z|},\quad \forall
\varepsilon>0,\ \forall z\in \mathbb{C}.
\end{equation*}
The smallest $\sigma=\sigma_{f}$ in this inequality is called a spherical type of $f$.
In other words, the class $B_{p, k}^\sigma$ is the collection of all entire
functions of spherical type at most $\sigma$.

We say that a function $f\in
\widetilde{B}_{p, k}^\sigma$ if $f\in L^{p}(\mathbb{R}^{d},d\mu_{k})$ is such
that its analytic continuation to $\mathbb{C}^d$ satisfies
\[
|f(z)|\leq c_{f}e^{\sigma|\mathrm{Im}\,z|},\quad \forall z\in \mathbb{C}^d.
\]

Both classes coincide \cite{iv3} and by the Paley--Wiener theorem for tempered distributions (see \cite{Jeu06,Tri02})
we get the following characterization.

\begin{proposition}[\cite{iv3}]\label{prop3.1}
A function $f\in B_{p, k}^\sigma$, $1\leq p<\infty$, iff $f\in
L^{p}(\mathbb{R}^{d},d\mu_{k})\cap C_b(\R^d)$ and
$\mathrm{supp}\,\mathcal{F}_k(f)\subset B_\sigma$.
\end{proposition}

The Dunkl transform $\mathcal{F}_k(f)$ in Proposition \ref{prop3.1} is understood
as a function for $1\leq p\leq 2$ and as a tempered distribution for $p>2$.

Let
\[
E_{\sigma}(f)_{p,d\mu_k}=\inf\{\|f-g\|_{p,d\mu_k}\colon g\in B_{p, k}^\sigma\}
\]
be the value of the best approximation of a function $f\in L^{p}(\mathbb{R}^{d},d\mu_{k})$
by entire functions of spherical exponential type at most $\sigma$.
The best approximation is achieved \cite{iv3}.


Now we define the fractional power of the Dunkl Laplacian. Let
\[
\Phi_k=\Bigl\{f\in \mathcal{S}(\mathbb{R}^d)\colon \int_{\mathbb{R}^{d}}x^{\alpha}f(x)\,d\mu_k(x)=0,\ \alpha\in \mathbb{Z}^d_+\Bigr\}
\]
be the weighted Lizorkin space (see, \cite{Liz63,iv4,SamKilMar93}),
\[
\Psi_k=\{\mathcal{F}_{k}(f)\colon f\in\Phi_k\}.
\]
At the spaces $\Phi_k$ and $\Psi_k$ we consider the same convergence as in $\mathcal{S}(\mathbb{R}^{d})$.

We proved \cite{iv4} that
\[
\Psi_k=\Psi_0=\{\mathcal{F}(f)\colon f\in\Phi_0\}=\{f\in\mathcal{S}(\mathbb{R}^{d})\colon D^{\alpha}f(0)=0,\,\, \alpha\in \mathbb{Z}^d_+\},
\]
where $D^{\alpha}f(x)=\prod_{j=1}^dD_j^{\alpha_j}f(x)$, $\alpha\in \mathbb{Z}^d_+$, and $D_jf(x)=D_{j,0}f(x)$ is the usual partial derivative with respect to a variable $x_j$, $j=1,\dots,d$.

The spaces $\Phi_k$ and $\Psi_k$ are closed. It is evidently for the space $\Psi_k$. If a sequence $\{\varphi_n\}\subset \Phi_k$
converges to $\varphi$ in $\Phi_k$ then $\varphi\in\mathcal{S}(\mathbb{R}^{d})$
and the orthogonality of $\varphi$ to the polynomial $x^n$ follows from estimation
\begin{align*}
\Bigl|\int_{\mathbb{R}^{d}}x^n\varphi(x)\,d\mu_k(x)\Bigr|&=\Bigl|\int_{\mathbb{R}^{d}}x^n(\varphi(x)-\varphi_n(x))\,d\mu_k(x)\Bigr|\\&
\lesssim\int_{\mathbb{R}^{d}}(1+|x|^2)^{-m}\,d\mu_k(x)\,\|(1+|x|^2)^{m+n/2}(\varphi(x)-\varphi_n(x))\|_{\infty},
\end{align*}
where $m>d_k/2$. The space $\Phi_k$ is dense in $L^{p}(\mathbb{R}^{d},d\mu_{k})$, $1\leq p<\infty$ \cite{iv4}.

Let $\Phi_k'$ and $\Psi_k'$ be the spaces of distributions
on $\Phi_k$ and $\Psi_k$ accordingly. We have $\mathcal{S}'(\mathbb{R}^{d})\subset\Phi_k'$, $\mathcal{S}'(\mathbb{R}^{d})\subset\Psi_k'$.
We can multiply distributions from $\Psi_k'$ on functions from $C^{\infty}_{\Pi}(\mathbb{R}^d\setminus\{0\})$.

\begin{lemma}\label{lem3.1}
If $g\in C^{\infty}_{\Pi}(\mathbb{R}^d\setminus\{0\})$, $f\in\Psi_k'$, then $gf \in\Psi_k'$, where
\[
\<gf,\psi\>=\<f,g\psi\>,\quad \psi\in\Psi_k.
\]
\end{lemma}

\begin{proof} It is necessary to prove that if a sequence $\{\psi_l\}\subset\Psi_k$ converges to zero in topology of
$\mathcal{S}(\mathbb{R}^{d})$, then the sequence $\{g\psi_l\}\subset\Psi_k$ converges to zero in topology of
$\mathcal{S}(\mathbb{R}^{d})$ too. We can prove it only for the sequence $\{|x|^r\psi_l\}$, since $g(x)=|x|^rg_1(x)$,
$g_1\in C^{\infty}_{\Pi}(\mathbb{R}^d)$, and the sequence $\{g_1\psi_l\}$ converges to zero in topology of
$\mathcal{S}(\mathbb{R}^{d})$. The topology of $\mathcal{S}(\mathbb{R}^{d})$ is generated by a countable family of norms
\[
p_{n,\alpha}(\varphi)=\sup_{x}(1+|x|^2)^{n/2}|D^{\alpha}\varphi(x)|,\quad \varphi\in \mathcal{S}(\mathbb{R}^{d}),
n\in\mathbb{Z}_+, \alpha\in\mathbb{Z}^d_+.
\]
It is known that this topology on $\Psi_k$ is equivalent to the topology defined by a countable family of norms
\begin{equation}\label{topology}
q_{n,\alpha,m}(\psi)=\sup_{x}(1+|x|^2)^{n/2}|x|^{-m}|D^{\alpha}\psi(x)|,\quad \psi\in \Psi_k, m\in\mathbb{Z}_+
\end{equation}
(see, \cite{SamKilMar93}). Using Liouville formula
\[
D^{\alpha}(|x|^r\psi(x))=\sum_{\beta\leq\alpha}c_{\alpha,\beta}D^{\beta}(|x|^r)D^{\alpha-\beta}\psi(x)=
\sum_{\beta\leq\alpha}c_{\alpha,\beta}\sum_{s=1}^{|\beta|_1}|x|^{r-2s}t_{s,\beta}(x)D^{\alpha-\beta}\psi(x),
\]
where $t_{s,\beta}(x)$ are polynomials of degree at most $s$, we can estimate $p_{n,\alpha}(|\Cdot|^r\psi)$ through
a finite sum of norms \eqref{topology}. Lemma~\ref{lem3.1} is proved.
\end{proof}

\begin{remark}
Next, using multipliers from $C^{\infty}_{\Pi}(\mathbb{R}^d\setminus\{0\})$ and Dunkl transform we define
several distributions. Since a sequence $\{\varphi_l\}\subset\Phi_k$ converges to zero in topology of
$\mathcal{S}(\mathbb{R}^{d})$, iff the sequence $\{\mathcal{F}_k(\varphi_l)\}\subset\Psi_k$ converges to zero
in topology of $\mathcal{S}(\mathbb{R}^{d})$, then by Lemma~\ref{lem3.1} all functionals will be continue.
\end{remark}

Let $f\in\mathcal{S}'(\mathbb{R}^{d})$. If $f(\Psi_k)=0$, then $\mathrm{supp}\,f=\{0\}$
and $f=\sum_{|\alpha|_1\leq N}c_{\alpha} D^{\alpha}\delta_0$, where $\<\delta_0,\varphi\>=\varphi(0)$.
Since $\mathcal{F}_k(\sum_{|\alpha|_1\leq N}c_\alpha D^{\alpha}\delta_0)\in\Pi$, then $f\in\Pi$ if $f(\Phi_k)=0$.
So, if $f=0$ in $\Phi_k'$ and
$f\notin \Pi$, then $f=0$ in $\mathcal{S}'(\mathbb{R}^{d})$.
We can assume that $\Phi_k'$ is a factor space $\Phi_k'=\mathcal{S}'(\mathbb{R}^{d})/\Pi$.
Further distributions from $\Phi_k'$, differing by the polynomial, we will not distinguish.
Analogously, $\Psi_k'=\mathcal{S}'(\mathbb{R}^{d})/\mathcal{F}_k(\Pi)$.

Let $r>0$. First we define the $r$-th power of the Dunkl Laplacian for $\varphi\in\Phi_k$
as follows
\[
(-\Delta_k)^r\varphi=\mathcal{F}_k^{-1}(|\Cdot|^{2r}\mathcal{F}_k(\varphi))
=\mathcal{F}_k(|\Cdot|^{2r}\mathcal{F}_k^{-1}(\varphi))\in\Phi_k.
\]
For $f\in\Phi_k'$ the distribution $(-\Delta_k)^rf\in\Phi_k'$ is defined  by relation
\[
\<(-\Delta_k)^rf,\varphi\>=\<f,(-\Delta_k)^r\varphi\>=
\<f,\mathcal{F}_k^{-1}(|\Cdot|^{2r}\mathcal{F}_k(\varphi))\>,\quad  \varphi\in \Phi_k.
\]

Let $W^{r}_{p, k}$, $1\leq p\leq\infty$, be the Sobolev space, that is,
\[
W^{r}_{p, k}=\{f\in L^{p}(\mathbb{R}^{d},d\mu_{k})\colon (-\Delta_k)^{r/2}f\in
L^{p}(\mathbb{R}^{d},d\mu_{k})\}
\]
equipped with the  Banach norm
\[
\|f\|_{W^{r}_{p,k}}=\|f\|_{p,d\mu_k}+\|(-\Delta_k)^{r/2}f\|_{p,d\mu_k}.
\]

Now for distributions we define direct and inverse Dunkl transforms $\mathcal{F}_k,\mathcal{F}_k^{-1}$,
generalized translation operators $\tau^y$, $T^t$, and convolution $(f\Ast{k}g)$.

For $f\in\Phi_k'$ direct Dunkl transform
$\mathcal{F}_k(f)\in\Psi_k'$ is defined by
\[
\<\mathcal{F}_k(f),\psi\>=\<f,\mathcal{F}_k(\psi)\>, \quad\psi\in\Psi_k.
\]
For $g\in\Psi_k'$ inverse Dunkl transform
$\mathcal{F}_k^{-1}(g)\in\Phi_k'$ is defined by
\[
\<\mathcal{F}_k^{-1}(g),\varphi\>=\<g,\mathcal{F}_k^{-1}(\varphi)\>, \quad\varphi\in\Phi_k.
\]
We have
\[
\mathcal{F}_k^{-1}(\mathcal{F}_k(f))=f,\quad\mathcal{F}_k(\mathcal{F}_k^{-1}(g))=g,
\quad f\in\Phi_k',\,\,g\in\Psi_k'.
\]

Note that $f=g$ in $\Phi_k'$ iff $\mathcal{F}_k(f)=\mathcal{F}_k(g)$ in $\Psi_k'$.

For $f\in\Phi_k'$ the generalized translation operators $\tau^yf,T^tf\in\Phi_k'$ are defined as follows
\begin{equation*}
\begin{gathered}
\<\tau^yf,\varphi\>=\<f,\tau^{-y}\varphi\>=\<f,\mathcal{F}_k^{-1}(e_k(-y,\Cdot)\mathcal{F}_k(\varphi))\>,
\quad \varphi\in \Phi_k,\quad y\in\mathbb{R}^{d},\\
\<T^tf,\varphi\>=\<f,T^t\varphi\>=\<f,\mathcal{F}_k^{-1}(j_{\lambda_k}(t|\Cdot|)\mathcal{F}_k(\varphi))\>,
\quad \varphi\in \Phi_k,\quad t\in\mathbb{R}_+.
\end{gathered}
\end{equation*}

For the Dunkl transform of the considered operators and their compositions we have the following easily verifiable equalities
\begin{equation}\label{distributiontransform}
\begin{gathered}
\mathcal{F}_k((-\Delta_k)^{r}f)=|\Cdot|^{2r}\mathcal{F}_k(f),\quad \mathcal{F}_k(\tau^yf)=e_k(y,\Cdot)\mathcal{F}_k(f),\\
\mathcal{F}_k((-\Delta_k)^{r}\tau^yf)=|\Cdot|^{2r}e_k(y,\Cdot)\mathcal{F}_k(f), \quad
\mathcal{F}_k(T^tf)=j_{\lambda_k}(t|\Cdot|)\mathcal{F}_k(f),\\
\mathcal{F}_k((-\Delta_k)^{r}T^tf)=|\Cdot|^{2r}j_{\lambda_k}(t|\Cdot|)\mathcal{F}_k(f),\\
\mathcal{F}_k(T^t(\tau^yf))=j_{\lambda_k}(t|\Cdot|)e_k(y,\Cdot)\mathcal{F}_k(f).
\end{gathered}
\end{equation}
This implies the commutativity of these compositions.

Let $\varphi\in\Phi_k$, $\check{\varphi}(y)=\varphi(-y)$. We call $f\in\Phi_k'$ even if
$\<f,\check{\varphi}\>=\<f,\varphi\>$. Even $g\in\Psi_k'$ is defined similarly. Note that
$f\in\Phi_k'$ is even iff $\mathcal{F}(f)\in\Psi_k'$ is even.

Let $N_k$ be a set of even $f\in\Phi_k'$ for which $\mathcal{F}_k(f)\in C^{\infty}_{\Pi}(\mathbb{R}^d\setminus\{0\})$.
For $f\in N_k$ and $\varphi\in\Phi_k$ we set
\[
(f\Ast{k}\varphi)(x)=\<\tau^{x}f,\check{\varphi}\>=\<f,\tau^{-x}\check{\varphi}\>.
\]

\begin{lemma}\label{lem3.2}
If $g\in N_k$, $\varphi\in\Phi_k$, then $(g\Ast{k}\varphi)\in\Phi_k$ and
\[ (g\Ast{k}\varphi)(x)=\mathcal{F}_k^{-1}(\mathcal{F}_k(g)\mathcal{F}_k(\varphi))(x),\quad
\mathcal{F}_k(g\Ast{k}\varphi)(y)=\mathcal{F}_k(g)(y)\mathcal{F}_k(\varphi)(y).
\]
\end{lemma}

\begin{proof} We have
\begin{align*}
\tau^{-x}\check{\varphi}(y)&=\mathcal{F}_k^{-1}(e_k(-x,\Cdot)\mathcal{F}_k(\check{\varphi}))(y)\\&
=\int_{\mathbb{R}^d}e_k(-x,z)e_k(y,z)\mathcal{F}_k(\check{\varphi})(z)\,d\mu_k(z)\\&
=\int_{\mathbb{R}^d}e_k(-x,z)e_k(y,z)\mathcal{F}_k(\varphi)(-z)\,d\mu_k(z)\\&
=\int_{\mathbb{R}^d}e_k(x,z)e_k(-y,z)\mathcal{F}_k(\varphi)(z)\,d\mu_k(z)\\&
=\mathcal{F}_k(e_k(x,\Cdot)\mathcal{F}_k(\varphi))(y)=\tau^{x}\varphi(-y)\in\Phi_k(\mathbb{R}^{d}).
\end{align*}
Hence, by definition
\begin{align*}
(g\Ast{k}\varphi)(x)&=\<g,\tau^{-x}\check{\varphi}\>=\<g,\mathcal{F}_k(e_k(x,\Cdot)\mathcal{F}_k(\varphi))\>
=\<\mathcal{F}_k(g),e_k(x,\Cdot)\mathcal{F}_k(\varphi)\>\\&
=\int_{\mathbb{R}^d}e_k(x,z)\mathcal{F}_k(g)(z)\mathcal{F}_k(\varphi)(z)\,d\mu_k(z)\\&
=\mathcal{F}_k^{-1}(\mathcal{F}_k(g)\mathcal{F}_k(\varphi))(x)\in\Phi_k(\mathbb{R}^{d}).
\end{align*}
Lemma~\ref{lem3.2} is proved.
\end{proof}

Using Lemma~\ref{lem3.2}, we can define a convolution $(f\Ast{k}g)\in\Phi_k'$ for $f\in\Phi_k'$ and $g\in N_k$ as follows
\begin{equation}\label{convolution2}
\<(f\Ast{k}g),\varphi\>=\<f,(g\Ast{k}\varphi)\,\>,\quad \varphi\in\Phi_k.
\end{equation}
By Lemma~\ref{lem3.3} and $\mathcal{F}_k(\mathcal{F}_k(\psi))=\check{\psi}$ for $\psi\in\Psi_k$, we obtain
\begin{align*}
\<\mathcal{F}_k(f\Ast{k}g),\psi\>&=\<(f\Ast{k}g),\mathcal{F}_k(\psi)\>=\<f,(g\Ast{k}\mathcal{F}_k(\psi))\>\\&
=\<f,\mathcal{F}_k^{-1}(\mathcal{F}_k(g)\mathcal{F}_k(\mathcal{F}_k(\psi)))\>=\<f,\mathcal{F}_k^{-1}(\mathcal{F}_k(g)\check{\psi})\>\\&
=\<f,\mathcal{F}_k(\mathcal{F}_k(g)\psi)\>=\<\mathcal{F}_k(f),\mathcal{F}_k(g)\psi\>=\<\mathcal{F}_k(g)\mathcal{F}_k(f),\psi\>,
\end{align*}
hence
\begin{equation}\label{transformconvolution2}
\mathcal{F}_k(f\Ast{k}g)=\mathcal{F}_k(g)\mathcal{F}_k(f).
\end{equation}

We give some simple properties of convolution. The distribution $|\Cdot|^r\in\Psi_k'$ is even,
$G_r=\mathcal{F}_k^{-1}(|\Cdot|^r)\in N_k$ and $(-\Delta)^{r/2}f=(f\Ast{k}G_r)$ for $f\in\Phi_k'$.
If $g_1,g_2\in N_k$, then $(g_1\Ast{k}g_2)\in N_k$ and $(g_1\Ast{k}g_2)=(g_2\Ast{k}g_1)$. If $f\in\Phi_k'$ and $g_1,g_2\in N_k$ then
\begin{equation*}
(f\Ast{k}(g_1\Ast{k}g_2))=((f\Ast{k}g_1)\Ast{k}g_2).
\end{equation*}

We have
\begin{equation*}
(-\Delta_k)^{r}(f\Ast{k}g)=((-\Delta_k)^{r}f\Ast{k}g).
\end{equation*}
Indeed, by \eqref{distributiontransform}, \eqref{transformconvolution2},
\begin{align*}
\mathcal{F}_k((-\Delta_k)^{r}(f\Ast{k}g))&=|\Cdot|^{2r}\mathcal{F}_k(f\Ast{k}g)=|\Cdot|^{2r}\mathcal{F}_k(f)\mathcal{F}_k(g)\\&
=\mathcal{F}_k((-\Delta_k)^{r}f)\mathcal{F}_k(g)=\mathcal{F}_k((-\Delta_k)^{r}f\Ast{k}g).
\end{align*}

\begin{lemma}\label{lem3.3}
If $r>0$, $1\leq p\leq\infty$, then $\mathcal{S}(\mathbb{R}^d)\subset W^{r}_{p, k}$. More exactly,
if $f\in\mathcal{S}(\mathbb{R}^d)$, then $\mathcal{F}_k(|\Cdot|^r\mathcal{F}_k(f))\in L^1(\mathbb{R}^d,d\mu_k)$
and $(-\Delta)^{r/2}f\in\mathcal{A}_k$.
\end{lemma}

\begin{proof} As $\mathcal{F}_k((-\Delta)^{r/2}f)(y)=|y|^r\mathcal{F}_k(f)(y)$=$|y|^rg(y)$,
where $g\in\mathcal{S}(\mathbb{R}^d)$ and $|y|^rg(y)\in L^1(\mathbb{R}^d,d\mu_k)\cap C_b(\mathbb{R}^d)$,
it is sufficient to show $\mathcal{F}_k(|\Cdot|^rg)\in L^1(\mathbb{R}^d,d\mu_k)$. We have
\[
|y|^r=(1+|y|^2)^r\Bigl(1-\frac{1}{1+|y|^2}\Bigr)^r=\sum_{s=0}^{\infty}(-1)^{s}\binom{r}{s}\frac{1}{(1+|y|^2)^{s-r}},
\]
where according to the complement formula and the asymptotic of the gamma-function as $s\to\infty$
\[
\binom{r}{s}=\frac{\Gamma(r+1)}{\Gamma(s+1)\Gamma(r-s+1)}=
\frac{\sin\pi(s-r)\Gamma(r+1)\Gamma(s-r)}{\pi\Gamma(s+1)}=O\Bigl(\frac{1}{s^{r+1}}\Bigr),
\]
and
\begin{equation}\label{const}
\sum_{s=1}^{\infty}\Bigl|\binom{r}{s}\Bigr|=c(r)<\infty.
\end{equation}

Let us consider the following decomposition
\[
|y|^rg(y)=\sum_{s\leq r}(-1)^{s}\binom{r}{s}\frac{g(y)}{(1+|y|^2)^{s-r}}+\sum_{s>r}(-1)^{s}\binom{r}{s}\frac{g(y)}{(1+|y|^2)^{s-r}}=g_1(y)+g_2(y).
\]

Since $g_1\in\mathcal{S}(\mathbb{R}^d)$, then $\mathcal{F}_k(g_1)\in L^1(\mathbb{R}^d,d\mu_k)$.
Further from \eqref{restriction} for any $\alpha>0$
\[
h_{\alpha}(y)=\mathcal{F}_k((1+|\Cdot|^2)^{-\alpha})(y)=\mathcal{H}_{\lambda_k}((1+(\Cdot)^2)^{-\alpha})(|y|).
\]
It is known \cite{Pla07} that $\mathcal{H}_{\lambda_k}((1+(\Cdot)^2)^{-\alpha})\in L^1(\mathbb{R}_+,d\nu_{\lambda_k})$ and
\[
\|h_{\alpha}\|_{1,d\mu_k}=
\|\mathcal{H}_{\lambda_k}((1+(\Cdot)^2)^{-\alpha})\|_{1,d\nu_{\lambda_k}}\leq 1.
\]
By Proposition~\ref{prop2.4} we obtain $\mathcal{F}_k((f\Ast{k}h_{\alpha}))(y)=(1+|y|^2)^{-\alpha}g(y)$ and
\[
\|\mathcal{F}_k((1+|\Cdot|^2)^{-\alpha}g)\|_{1,d\mu_k}\leq \|h_{\alpha}\|_{1,d\mu_k}\|f\|_{1,d\mu_k}\leq\|f\|_{1,d\mu_k},
\]
therefore from~\eqref{const}
\[
\|\mathcal{F}_k(g_2)\|_{1,d\mu_k}\leq\sum_{s>r}\Bigl|\binom{r}{s}\Bigr|\|h_{s-r}\|_{1,d\mu_k}\|f\|_{1,d\mu_k}\leq c(r)\|f\|_{1,d\mu_k}.
\]
Lemma~\ref{lem3.3} is proved.
\end{proof}

\begin{lemma}\label{lem3.4}
If $f\in L^{p}(\mathbb{R}^{d},d\mu_{k})$, $g\in L^{1}_{\mathrm{rad}}(\mathbb{R}^d,d\mu_{k})$, $\mathcal{F}_k(g)\in N_k$, then both convolutions
\eqref{convolution1} and \eqref{convolution2} of these functions coincide.
\end{lemma}

\begin{proof} Set
\[
(f\Ast{k}g)(x)=\int_{\mathbb{R}^d}f(y)\tau^{x}g(-y)\,d\mu_{k}(y).
\]
By Proposition~\ref{prop2.4} $(f\Ast{k}g)\in L^{p}(\mathbb{R}^{d},d\mu_{k})$ and $(f\Ast{k}g)\in \Phi_k'$.
It is sufficiently to prove the equality $\mathcal{F}_k(f\Ast{k}g)=\mathcal{F}_k(g)\mathcal{F}_k(f)$ in $\Psi_k'$. For any $\psi\in \Psi_k$
we have
\begin{align*}
\<\mathcal{F}_k(f\Ast{k}g),\psi\>&=\<(f\Ast{k}g),\mathcal{F}_k(\psi)\>\\&=
\int_{\mathbb{R}^d}\int_{\mathbb{R}^d}f(y)\tau^{x}g(-y)\,d\mu_{k}(y)\mathcal{F}_k(\psi)(x)\,d\mu_{k}(x)\\&
=\int_{\mathbb{R}^d}f(y)\int_{\mathbb{R}^d}\tau^{-y}g(x)\mathcal{F}_k(\psi)(x)\,d\mu_{k}(x)d\mu_{k}(y).
\end{align*}
Since
\begin{align*}
&\int_{\mathbb{R}^d}\tau^{-y}g(x)\mathcal{F}_k(\psi)(x)\,d\mu_{k}(x)\\&\qquad =
\int_{\mathbb{R}^d}\int_{\mathbb{R}^d}e_k(-y,z)e_k(x,z)\mathcal{F}_k(g)(z)\,d\mu_{k}(z)\mathcal{F}_k(\psi)(x)\,d\mu_{k}(x)
\\&\qquad =\int_{\mathbb{R}^d}e_k(-y,z)\mathcal{F}_k(g)(z)\psi(z)\,d\mu_{k}(z)=\mathcal{F}_k(\mathcal{F}_k(g)\psi)(y),
\end{align*}
then
\begin{align*}
\<\mathcal{F}_k(f\Ast{k}g),\psi\>&=\int_{\mathbb{R}^d}f(y)\mathcal{F}_k(\mathcal{F}_k(g)\psi)\,d\mu_{k}
=\<f,\mathcal{F}_k(\mathcal{F}_k(g)\psi)\>
\\&=\<\mathcal{F}_k(f),\mathcal{F}_k(g)\psi\>=\<\mathcal{F}_k(g)\mathcal{F}_k(f),\psi\>.
\end{align*}
Lemma~\ref{lem3.4} is proved.
\end{proof}




Define the $K$-functional for the couple $(L^{p}(\mathbb{R}^{d},d\mu_{k}), W^{r}_{p, k})$
as follows
\[
K_{r}(t,
f)_{p,d\mu_k}=\inf\{\|f-g\|_{p,d\mu_k}+t^{r}\|(-\Delta_k)^{r/2}g\|_{p,d\mu_k}\colon
g\in W^{r}_{p, k}\}.
\]
Note that
for any $f_1,f_2\in L^{p}(\mathbb{R}^{d},d\mu_{k})$ and $g\in W^{r}_{p, k}$, we have
\begin{align*}
&\|f_1-g\|_{p,d\mu_k}+t^{r}\|(-\Delta_k)^rg\|_{p,d\mu_k}\\
&\qquad \leq \|f_2-g\|_{p,d\mu_k}+t^{r}\|(-\Delta_k)^{r/2}g\|_{p,d\mu_k}+\|f_1-f_2\|_{p,d\mu_k}
\end{align*}
and hence,
\begin{equation}\label{K-inequality}
|K_{r}(t, f_1)_{p,d\mu_k}-K_{r}(t, f_2)_{p,d\mu_k}|\leq \|f_1-f_2\|_{p,d\mu_k}.
\end{equation}
If $f\in W^{r}_{p, k}$, then $K_{r}(t, f)_{p,d\mu_k}\leq t^{r}\|(-\Delta_k)^{r/2}f\|_{p,d\mu_k}$ and
$\lim_{t\to 0}K_{r}(t, f)_{p,d\mu_k}=0$. This, \eqref{K-inequality} and Lemma~\ref{lem3.3}
imply that, for any $f\in L^{p}(\mathbb{R}^{d},d\mu_{k})$,
\begin{equation}\label{K-property1}
\lim_{t\to 0}K_{r}(t, f)_{p,d\mu_k}=0.
\end{equation}
Another important property of the $K$-functional is
\begin{equation}\label{K-property2}
K_{r}(\lambda t, f)_{p,d\mu_k}\leq \max\{1,\,\lambda^{r}\}K_{r}(t, f)_{p,d\mu_k}.
\end{equation}

Let $I$ be an identical operator and $m>0$. Consider the following difference
\begin{equation}\label{difference}
\varDelta_t^mf(x)=(I-T^t)^{m/2}f(x)= \sum_{s=0}^{\infty}(-1)^s\binom{\frac{m}{2}}{s}(T^{t})^sf(x).
\end{equation}
Difference \eqref{difference} coincide with the classical fractional
difference for the translation operator $T^{t}f(x)=f(x+t)$ and correspond to
the usual definition of the fractional modulus of smoothness.

The modulus of smoothness of order $m$ of a function $f\in L^{p}(\mathbb{R}^{d},d\mu_{k})$ is defined~by
\begin{equation*}
\omega_m(\delta, f)_{p,d\mu_k}=\sup_{0<t\leq\delta}\|\varDelta_t^mf(x)\|_{p,d\mu_k},
\end{equation*}

Let us mention some basic properties of this modulus of smoothness.
Using the triangle inequality, Proposition~\ref{prop2.3} and \eqref{const}, we obtain
\begin{equation}\label{ineqmodulus}
\begin{gathered}
\omega_m(\delta, f_1+f_2)_{p,d\mu_k}\leq \omega_m(\delta, f_1)_{p,d\mu_k}+\omega_m(\delta, f_2)_{p,d\mu_k},\\
 \omega_m(\delta, f)_{p,d\mu_k}\leq c(m/2)\|f\|_{p,d\mu_k}.
\end{gathered}
\end{equation}

Let $t\in\mathbb{R}_+$,
\begin{equation*}
j_{\lambda_{k},m}(t)=\sum_{s=0}^{\infty}(-1)^s\binom{\frac{m}{2}}{s}\bigl(j_{\lambda_{k}}(t)\bigr)^s=(1-j_{\lambda_{k}}(t))^{m/2},
\end{equation*}
where $\lambda_k$ is defined in \eqref{Weighteddimension}. Since $j_{\lambda_{k},m}(|y|)\in C^{\infty}_{\Pi}(\mathbb{R}^d\setminus\{0\})$,
then for $\varphi\in \Phi_k$ by \eqref{multiplier} and \eqref{difference},
\begin{equation*}
\mathcal{F}_k(\varDelta_t^m\varphi)(y)=j_{\lambda_{k},m}(t|y|)\mathcal{F}_k(\varphi)(y)\in \Psi_k,
\end{equation*}
and $\varDelta_t^m\varphi\in \Phi_k$. Hence, for $f\in \Phi_k'$ we can define distributions $\varDelta_t^mf\in \Phi_k'$
and $\mathcal{F}_k(\varDelta_t^mf)\in \Psi_k'$ by equalities
\begin{equation*}
\begin{gathered}
\<\varDelta_t^mf,\varphi\>=\<f,\varDelta_t^m\varphi\>=\<f,\mathcal{F}_k^{-1}(j_{\lambda_k,m}(t|\Cdot|)\mathcal{F}_k(\varphi))\>,
\quad \varphi\in \Phi_k,\\
\<\mathcal{F}_k(\varDelta_t^mf),\psi\>=\<f,\varDelta_t^m\mathcal{F}_k(\psi)\>=\<f,\mathcal{F}_k(j_{\lambda_k,m}(t|\Cdot|)\psi)\>,
\quad \psi\in \Psi_k.
\end{gathered}
\end{equation*}
In the last definition we used the equality
\[
\mathcal{F}_k^{-1}(j_{\lambda_k,m}(t|\Cdot|)\mathcal{F}_k(\varphi))(x)
=\mathcal{F}_k(j_{\lambda_k,m}(t|\Cdot|)\mathcal{F}_k^{-1}(\varphi))(x).
\]

Since

\[
\<f,\mathcal{F}_k(j_{\lambda_k,m}(t|\Cdot|)\psi)\>=\<\mathcal{F}_k(f),j_{\lambda_k,m}(t|\Cdot|)\psi\>=
\<j_{\lambda_k,m}(t|\Cdot|)\mathcal{F}_k(f),\psi\>,
\]
then
\begin{equation}\label{transformdifference1}
\mathcal{F}_k(\varDelta_t^mf)=j_{\lambda_k,m}(t|\Cdot|)\mathcal{F}_k(f),\quad
\end{equation}

Applying \eqref{distributiontransform}, we obtain
\begin{equation}\label{transformdifference2}
\mathcal{F}_k(\varDelta_t^m(-\Delta_k)^{r/2}f))=\mathcal{F}_k((-\Delta_k)^{r/2}\varDelta_t^mf))
=|\Cdot|^rj_{\lambda_k,m}(t|\Cdot|)\mathcal{F}_k(f).
\end{equation}

The function $j_{\lambda_{k},m}(t)$ has zero of order $m$ at the origin.
Indeed, applying the expansion
\begin{equation*}
j_{\lambda}(t)=\sum_{k=0}^{\infty}\frac{(-1)^{k}\Gamma(\lambda+1)(t/2)^{2k}}{k!\,\Gamma(k+\lambda+1)},
\end{equation*}
we get $j_{\lambda_{k},m}(t)\asymp t^{m}$ as $t\to 0$.

\begin{lemma}\label{lem3.5}
If $m,r>0$, $1\leq p\leq\infty$, and $f\in L^{p}(\mathbb{R}^{d},d\mu_{k})$, then
\begin{equation}\label{omega-inequality2}
\omega_{m+r}(\delta, f)_{p,d\mu_k}\leq c(r/2)\omega_{m}(\delta, f)_{p,d\mu_k}.
\end{equation}
\end{lemma}

\begin{proof} Since for $f\in\Phi_k'$ by~\eqref{distributiontransform}, \eqref{transformdifference1},
\begin{align*}
\mathcal{F}_k(\varDelta_t^{m+r}f)&=(1-j_{\lambda_{k}}(t))^{(m+r)/2}\mathcal{F}_k(f)\\&
=(1-j_{\lambda_{k}}(t))^{r/2}(1-j_{\lambda_{k}}(t))^{m/2}\mathcal{F}_k(f)\\&
=\sum_{s=0}^{\infty}(-1)^s\binom{\frac{r}{2}}{s}(j_{\lambda_{k}}(t))^s\mathcal{F}_k(\varDelta_t^{m}f),
\end{align*}
then
\[
\varDelta_t^{m+r}f=\sum_{s=0}^{\infty}(-1)^s\binom{\frac{r}{2}}{s}(T^{t})^s\varDelta_t^{m}f.
\]
Using for $f\in L^{p}(\mathbb{R}^{d},d\mu_{k})$ Proposition~\ref{prop2.3}, we get
\[
\|\varDelta_t^{m+r}f\|_{p,d\mu_k}\leq c(r/2)\|\varDelta_t^{m}f\|_{p,d\mu_k}.
\]
Lemma~\ref{lem3.5} is proved.

\end{proof}

\section{Jackson's inequality and equivalence of modulus of smoothness and $K$-functional}

\subsection{Main results}
First we state the Jackson-type inequality.

\begin{theorem}\label{thm4.1}
Let $\sigma> 0$, $1\leq p \leq \infty$, $r\geq 0$, $m>0$. We have,
for any $f\in W_{p, k}^{r}$,
\begin{equation}\label{Jackson}
E_{\sigma}(f)_{p,d\mu_k} \lesssim
\frac{1}{\sigma^{r}}\,\omega_m\Bigl(\frac{1}{\sigma},
(-\Delta_k)^{r/2}f\Bigr)_{p,d\mu_k}.
\end{equation}

\end{theorem}

\begin{remark}\label{rem4.1}
From the proof of Theorem
\ref{thm4.1} we will see that inequality \eqref{Jackson} can be equivalently written as
\[
E_{\sigma}(f)_{p,d\mu_k}
\lesssim
\frac{1}{\sigma^{r}}
\|\varDelta_{1/\sigma}^m ((-\Delta_k)^{r/2}f) \|_{p,d\mu_k},
\]
\end{remark}

The next theorem provides an equivalence between modulus of smoothness and the $K$-functional.
\begin{theorem}\label{thm4.2}
If $\delta,r> 0$, $1\leq p\leq \infty$, then for any $f\in L^{p}(\mathbb{R}^{d},d\mu_{k})$
\begin{equation}\label{K-equiv}
K_{r}(\delta, f)_{p,d\mu_k}\asymp \omega_r(\delta, f)_{p,d\mu_k}\asymp
\|\varDelta_\delta^r f \|_{p,d\mu_k}.
\end{equation}
\end{theorem}

We follow the proofs in \cite{iv3}. We will treat functions from $L^{p}(\mathbb{R}^{d},d\mu_{k})$
as distributions from $\Phi_k'$ and will use material of section~3.
Although the elements of $\Phi_k'$ are equivalence classes,
in each equivalence class there is only one element from
$L^{p}(\mathbb{R}^{d},d\mu_{k})$.

\subsection{Properties of the de la Vall\'{e}e Poussin type operators }

Let $\eta\in \mathcal{S}(\mathbb{R}^d)$ be such that $\eta(x)=1$ if $|x|\leq 1$,
$\eta(x)>0$ if $|x|<2$, and $\eta(x)=0$ if $|x|\geq 2$. Denote
\[
\eta_r(x)=\frac{1-\eta(x)}{|x|^{r}},\quad \widehat{\eta}_{k,r}(y)=\mathcal{F}_{k}(\eta_r)(y).
\]
We have $\eta_r\in C^{\infty}_{\Pi}(\mathbb{R}^d)$,
$\eta_r\in\Psi_k$, $\widehat{\eta}_{k,r}\in N_k$.
If $t=|x|$, $\eta_0(t)=\eta(x)$, and $\eta_{r 0}(t)=\eta_r(x)$, then by \eqref{restriction}
$\mathcal{F}_{k}(\eta_r)(y)=\mathcal{H}_{\lambda_k}(\eta_{r 0})(|y|)$.

\begin{lemma}[\cite{iv3}]\label{lem4.3}
We have $\widehat{\eta}_{k,r}\in L^{1}(\mathbb{R}^d, d\mu_k)$, where $r>0$.
\end{lemma}

For $m, r>0$ and $m\geq r$, we set
\[
g_{m, r}^*(y):= |y|^{-r}j_{\lambda_{k},m}(|y|),\quad
g_{m,r}(x):=\mathcal{F}_{k}(g_{m, r}^*)(x),
\]
\[
g_{m,r}^t(x) :=t^{r-2\lambda_k-2}g_{m,r}\Bigl(\frac{x}{t}\Bigr).
\]
Since
\[
g_{m,r}^*(y)=j_{\lambda_{k},m}(|y|)\eta_r(y)+|y|^{m-r}\frac{j_{\lambda_{k},m}(|y|)}{|y|^{m}}\,\eta(y),
\]
\[
\frac{j_{\lambda_{k},m}(|y|)}{|y|^{m}}\in
C_{\Pi}^{\infty}(\mathbb{R}^d),\quad\frac{j_{\lambda_{k},m}(|y|)}{|y|^{m}}\,\eta(y)\in
\mathcal{S}(\mathbb{R}^d),
\]
and by \eqref{distributiontransform}
\[
\mathcal{F}_{k}(j_{\lambda_{k},m}\eta_r)(x)=\sum_{s=0}^{\infty}(-1)^s\binom{m}{s}(T^{1})^s\widehat{\eta}_{\lambda_k,r}(x),
\]
then boundedness of the operator $T^1$ in $L^{1}(\mathbb{R}^d,d\mu_k)$ and Lemmas
\ref{lem3.3}, \ref{lem4.3} imply that
\begin{equation}\label{kernel}
\begin{gathered}
g_{m,r},\,g_{m,r}^t\in L^{1}(\mathbb{R}^d,d\mu_{k}),\quad \|g_{m,r}^t\|_{1,d\mu_{k}}=t^{r}\|g_{m,r}\|_{1,d\mu_{k}},
\\
\mathcal{F}_{k}^{-1}(g_{m,r}^t)(y)=
\mathcal{F}_{k}(g_{m,r}^t)(y)=t^{r}g_{m,r}^*(ty)=|y|^{-r}j_{\lambda_{k},m}(t|y|).
\end{gathered}
\end{equation}

\begin{lemma}\label{lem4.4}
Let $m,r>0$, $m \geq r$, $1\leq p\leq\infty$, and $f\in
W_{p, k}^{r}$. We have
\begin{equation}\label{differencerepres}
\Delta_t^mf = ((-\Delta_k)^rf\Ast{k}g_{m,r}^t)
\end{equation}
and
\begin{equation}\label{normdifferencerepres}
\|\Delta_t^mf\|_{p,d\mu_k}\lesssim t^{r}\|(-\Delta_k)^{r/2}f\|_{p,d\mu_k}.
\end{equation}
\end{lemma}

\begin{proof}
Applying \eqref{distributiontransform}, \eqref{transformconvolution2}, \eqref{transformdifference1},
and \eqref{kernel}, we obtain that for $f\in\Phi_k'$,
\begin{align*}
\mathcal{F}_k(\Delta_t^mf)(y)&=j_{\lambda_{k},m}(t|y|)\mathcal{F}_k(f)(y)
=|y|^{r}\mathcal{F}_k(f)(y)\frac{j_{\lambda_{k},m}(t|y|)}{|y|^{r}}\\
&=\mathcal{F}_k((-\Delta_k)^rf)(y)
\mathcal{F}_k(g_{m,r}^t)(y)=\mathcal{F}_k((-\Delta_k)^rf\Ast{k}g_{m,r}^t),
\end{align*}
and the equality \eqref{differencerepres} is fulfilled. If $f\in W_{p, k}^{r}$,
then $(-\Delta_k)^{r/2}f\in L^{p}(\mathbb{R}^{d},d\mu_{k})$ and
the inequality \eqref{normdifferencerepres} follows from \eqref{kernel},
\eqref{differencerepres}, Lemma~\ref{lem3.4} and Proposition \ref{prop2.4}.
Lemma~\ref{lem4.4} is proved.
\end{proof}

Let $f\in L^{p}(\mathbb{R}^d,d\mu_k)$. We set $\theta(x)=\mathcal{F}_k(\eta)(x)$ and
$\theta_\sigma(x) = \theta(x/\sigma)$. Then $\theta$, $\theta_\sigma\in
\mathcal{S}(\mathbb{R}^d)$.
The de la Vall\'{e}e Poussin type operator is given by $P_\sigma (f)= (f\Ast{k}\theta_\sigma)$.
By \eqref{transformconvolution2},
\[
\mathcal{F}_k(P_\sigma(f))(y)=\eta(y/\sigma)\mathcal{F}_k (f)(y).
\]

\begin{lemma}[\cite{iv3}]\label{lem4.5}
If $\sigma >0$, $1 \leq p\leq\infty$, $f\in L^{p}(\mathbb{R}^d,d\mu_k)$, then

\smallbreak
\textup{(1)} $P_\sigma (f)\in B_{p, k}^{2\sigma}$ and $P_\sigma(g)=g$ for any $g\in B_{p, k}^\sigma$;

\smallbreak
\textup{(2)} $\|P_\sigma (f)\|_{p,d\mu_{k}} \lesssim \|f\|_{p,d\mu_{k}}$;

\smallbreak
\textup{(3)} $\|f-P_\sigma(f)\|_{p,d\mu_{k}} \lesssim E_{\sigma} (f)_{p,d\mu_{k}}$.
\end{lemma}

Let $\lambda\geq-1/2$, $t\in\mathbb{R}_+$, $r\geq 0$,
\[
B_{\lambda}f(t)=f''(t)+\frac{2\lambda+1}{t}f'(t)
\]
be the Bessel differential operator. It is a restriction of $\Delta_k$ on radial functions.

In the proof of the next lemma we will use the estimates
\begin{equation}\label{besselestimate1}
B^s_{\lambda}((\Cdot)^{-r}(j_{\lambda})^l)(t)\lesssim l^{2s}t^{-l(\lambda+1/2)-r},\quad t\geq 1,\,\, r\geq 0,
\,\, s\in\mathbb{N},\,\, l>2s.
\end{equation}
The constant in \eqref{besselestimate1} don't depend from $t$ and $l$.

Let us prove \eqref{besselestimate1}. By induction we derive
\begin{equation}\label{derivative}
\begin{gathered}
B^s_{\lambda}((\Cdot)^{-r}(j_{\lambda})^l)(t)=t^{-2s-r}\sum_{m=1}^{2s}p_m(\lambda,r,s)t^m(j_{\lambda}(t))^l)^{(m)},\\
((j_{\lambda}(t))^l)^{(m)}=\sum_{i=1}^m(j_{\lambda}(t))^{l-i}\sum_{\Omega}
a_{\alpha,\beta}(l)(D^{\alpha}(j_{\lambda}(t)))^{\beta},
\end{gathered}
\end{equation}
where
\[
\alpha,\beta\in\mathbb{Z}_+^m, (D^{\alpha}(j_{\lambda}(t))^{\beta}=\prod_{j=1}^m((j_{\lambda}(t))^{(\alpha_j)})^{\beta_j},
\Omega=\{\alpha,\beta\colon \sum_{j=1}^m\alpha_j\beta_j=m,\sum_{j=1}^m\beta_j=i\},
\]
and $a_{\alpha,\beta}(l)$ are polynomials in $l$ of degree at most $m$. For derivatives of the Bessel
function we have the following estimates
\begin{equation}\label{besselestimate2}
|j^{(n)}_{\lambda}(t)|\lesssim(t+1)^{-(\lambda+1/2)},\quad t\in \mathbb{R}_+,\,\, n\in \mathbb{Z}_+,
\end{equation}
which follow, by induction on $n$, from the known properties of the Bessel function (\cite{BatErd53})
\[
|j_{\lambda}(t)|\lesssim(t+1)^{-(\lambda+1/2)},\quad
j'_{\lambda}(t)=-\frac{t}{2(\lambda+1)}j_{\lambda+1}(t).
\]
Substituting these estimates into \eqref{derivative}, we obtain \eqref{besselestimate1}.

\begin{lemma}\label{lem4.6}
If $\sigma >0$, $1 \leq p\leq\infty$, $m>0$, $r\geq 0$, $f\in W_{p, k}^{r}$, then
\begin{equation*}
\|f-P_{\sigma/2}(f)\|_{p,d\mu_{k}} \lesssim \sigma^{-r}\|\Delta_{1/\sigma}^m((-\Delta_k)^{r/2}f)\|_{p,d\mu_k}.
\end{equation*}
\end{lemma}

\begin{proof}
Let $f\in\Phi_k'$. Using \eqref{transformdifference2}, we get
\begin{align}
\mathcal{F}_k(f-P_{\sigma/2}(f))&=(1-\eta(2\Cdot/\sigma))\mathcal{F}_k(f)\notag\\
&=\sigma^{-r}\frac{1-\eta(2\Cdot/\sigma)}{(|\Cdot|/\sigma)^{r}
j_{\lambda_{k},m}(|\Cdot|/\sigma)}\mathcal{F}_k(\Delta_{1/\sigma}^m
(-\Delta_k)^{r/2}f)\notag\\ &=\sigma^{-r}\varphi(y/\sigma)
\mathcal{F}_k(\Delta_{1/\sigma}^m (-\Delta_k)^{r/2}
f),\label{transformapproximant}
\end{align}
where
\begin{equation*}
\varphi(y)= \frac{1-\eta(2y)}{|y|^{r} j_{\lambda_{k},m}(|y|)}.
\end{equation*}
Therefore
\[
f-P_{\sigma/2}(f)=\sigma^{-r}(\Delta_{1/\sigma}^m (-\Delta_k)^{r/2}f\Ast{k}\mathcal{F}_k(\varphi(\Cdot/\sigma))).
\]

We have $(-\Delta_k)^{r/2}f,\,\Delta_{1/\sigma}^m (-\Delta_k)^{r/2}f\in L^{p}(\mathbb{R}^d,d\mu_k)$. If
$\mathcal{F}_k(\varphi)\in L^{1}(\mathbb{R}^d,d\mu_k)$, then $\|\mathcal{F}_k(\varphi)\|_{1,d\mu_{k}}=
\|\mathcal{F}_k(\varphi(\Cdot/\sigma))\|_{1,d\mu_{k}}$
and by Lemma~\ref{lem3.4} and Proposition~\ref{prop2.4}
\[
\|f-P_{\sigma/2}(f)\|_{p,d\mu_{k}} \leq \sigma^{-r}\|\mathcal{F}_k(\varphi)\|_{1,d\mu_{k}}
\|\Delta_{1/\sigma}^m((-\Delta_k)^{r/2}f)\|_{p,d\mu_k}.
\]
It remains to prove the inclusion $\mathcal{F}_k(\varphi)\in L^{1}(\mathbb{R}^d,d\mu_k)$. Note that
$\varphi\in C^{\infty}_{\Pi}(\mathbb{R}^d)$, $\varphi\in\mathcal{S}'(\mathbb{R}^d)$
and $\mathcal{F}_k(\varphi)\in N_k$.

Using the expansion
\begin{equation}\label{expansion}
(1-t)^{-m/2}=\sum_{l=0}^{\infty}a_lt^l,\quad a_l=\frac{\Gamma(l+m/2)}{\Gamma(l+1)\Gamma(m/2)}\lesssim (l+1)^{m/2-1},
\end{equation}
we can write the following decomposition
\begin{equation}\label{decomposition}
\varphi(y)=\varphi_1(|y|)+\varphi_2(|y|),
\end{equation}
where
\[
\varphi_1(|y|)=2^{r}\eta_{r}(2y)\bigl((1-j_{\lambda_{k}}(|y|)^{-m/2}-S_N(j_{\lambda_{k}}(|y|))\bigr)\in C^{\infty}_{\Pi}(\mathbb{R}^d)
\]
and
\[
\varphi_2(|y|)=2^{r}\eta_{r}(2y)S_N(j_{\lambda_{k}}(|y|))\in C^{\infty}_{\Pi}(\mathbb{R}^d), \,\,
\eta_{r}(y)=\frac{1-\eta(y)}{|y|^{r}}, \,\,
S_N(t)=\sum_{l=0}^{N-1}a_lt^l.
\]

First, we show that $\mathcal{F}_k(\varphi_1(|\Cdot|))\in L^{1}(\mathbb{R}^d, d\mu_k)$.
Since for a radial function we have $\Delta_k\varphi_1(|y|)=B_{\lambda_k}(\varphi_{1})(|y|)$
then, by \eqref{besselestimate1} and \eqref{expansion}, we obtain
\[
|\Delta_k^s\varphi_1(|y|)|\lesssim \sum_{l=N}^{\infty}l^{2s+m-1} |y|^{-l(\lambda_k+1/2)}\lesssim
 |y|^{-N(\lambda_k+1/2)},\quad |y|\geq 2,\ s\in \Z_+.
\]
Hence, for a fixed $N\geq 2+2/(2\lambda_k+1)$, we have $\Delta_k^s\varphi_1(|y|)\in
L^{1}(\mathbb{R}^d, d\mu_k)$, where $s\in \mathbb{Z}_+$. Applying the equality
$\mathcal{F}_k((-\Delta)^s\varphi_1(|\Cdot|))(x)=|x|^{2s}\mathcal{F}_k(\varphi_1(|\Cdot|))(x)$, we derive that
\[
|\mathcal{F}_k(\varphi_1(|\Cdot|))(x)|\leq\frac{\|(-\Delta_k)^s \varphi_1(|y|)\|_{1,d\mu_k}}{|x|^{2s}}.
\]
Setting $s>\lambda_k+1$ yields $\mathcal{F}_k(\varphi_1(|\Cdot|))\in L^{1}(\mathbb{R}^d, d\mu_k)$.

Second, let us show that $\mathcal{F}_k(\varphi_2(|\Cdot|))\in L^{1}(\mathbb{R}^d, d\mu_k)$ for $r>0$.

Since $\varphi_2(|y|)=2^r\sum_{l=0}^{N-1}a_lj_{\lambda_k}^l(|y|)\eta_r(2y)$ and by \eqref{distributiontransform}
\[
\mathcal{F}_k(\varphi_2(|\Cdot|))(x)=2^r\sum_{l=0}^{N-1}a_l(T^1)^l\mathcal{F}_k(\eta_r(2\Cdot))(x),
\]
then boundedness of the operator $T^1$ in $L^{1}(\mathbb{R}^d,d\mu_k)$ and Lemma
\ref{lem4.3} imply that $\mathcal{F}_k(\varphi_2(|\Cdot|))\in L^{1}(\mathbb{R}^d, d\mu_k)$:
\[
\|\mathcal{F}_k(\varphi_2(|\Cdot|))\|_{1,d\mu_k}\leq 2^r\Bigl(\sum_{l=0}^{N-1}a_l\Bigr)\|\mathcal{F}_k(\eta_r(2\Cdot))\|_{1,d\mu_k}.
\]
For $r>0$ Lemma~\ref{lem4.6} is proved.

Let now $r=0$. Unfortunately, $\mathcal{F}_k(\varphi_2(|\Cdot|))\notin L^{1}(\mathbb{R}^d, d\mu_k)$.
We proceed as follows.  Using decomposition \eqref{decomposition} we define two operators $A_1$ and $A_2$ as follows:
\[
\mathcal{F}_k(A_1g)=\varphi_1(|\Cdot|/\sigma)\mathcal{F}_k(g),\quad
\mathcal{F}_k(A_2g)=\varphi_2(|\Cdot|/\sigma)\mathcal{F}_k(g).
\]
In accordance with \eqref{transformapproximant} it is sufficiently to prove that these operators are bounded in $L^{p}(\mathbb{R}^d, d\mu_k)$.

Since $\mathcal{F}_k(\varphi_1(|\Cdot|))\in L^{1}(\mathbb{R}^d, d\mu_k)$
and $A_1g=(f\Ast{k}\mathcal{F}_k(\varphi_1(|\Cdot|/\sigma)))$, then by Proposition~\ref{prop2.4}
for $g\in L^{p}(\mathbb{R}^d, d\mu_k)$
\[
\|A_1g\|_{p,d\mu_k}\leq
\|\mathcal{F}_k(\varphi_1(|\Cdot|))\|_{1,d\mu_k}\,\|g\|_{p,d\mu_k}\lesssim
\|g\|_{p,d\mu_k}.
\]

If
\[
Bg=\sum_{l=0}^{N-1}a_l(T^1)^lg,
\]
then by \eqref{distributiontransform} we have
\begin{align*}
\mathcal{F}_k(A_2g)&=(1-\eta(2\Cdot/\sigma))S_N(j_{\lambda_{k}}(|\Cdot|)\mathcal{F}_k(g)\\
&=(1-\eta(2y/\sigma))\mathcal{F}_k(Bg)
=\mathcal{F}_k(Bg-P_{\sigma/2}(Bg)),
\end{align*}
and $A_2g=Bg-P_{\sigma/2}(Bg)$. Applying boundedness of the operator $T^1$
in $L^{p}(\mathbb{R}^d,d\mu_k)$, we obtain
\[
\|Bg\|_{p,d\mu_k}\leq \Bigl(\sum_{l=0}^{N-1}a_l\Bigr)\|g\|_{p,d\mu_k},
\]
and by Lemma~\ref{lem4.5}
\[
\|A_2g\|_{p,d\mu_k}\leq \|Bg\|_{p,d\mu_k}+\|P_{\sigma/2}(Bg)\|_{p,d\mu_k}
\lesssim \|g\|_{p,d\mu_k}
\]
For $r=0$ Lemma~\ref{lem4.6} also is proved.
\end{proof}

\begin{lemma}\label{lem4.7}
If $\sigma >0$, $1 \leq p\leq\infty$, $m>0$, $f\in L^{p}(\mathbb{R}^d,d\mu_k)$, then
\begin{equation}\label{normapproximant2}
\|((-\Delta_k)^{m/2}P_\sigma(f)\|_{p,d\mu_k} \lesssim \sigma^{m}
\|\Delta_{1/(2\sigma)}^{m} f\|_{p,d\mu_k}.
\end{equation}
\end{lemma}

\begin{proof}
Applying \eqref{distributiontransform}, \eqref{transformconvolution2}, and \eqref{transformdifference1},
we obtain that for $f\in\Phi_k'$,
\begin{align*}
\mathcal{F}_k((-\Delta_k)^{m/2}P_\sigma(f))&=|\Cdot|^{m}\eta(\Cdot/\sigma)\mathcal{F}_k(f)\\
&=\sigma^{m}\varphi(\Cdot/\sigma)j_{\lambda_{k},m}(|\Cdot|/(2\sigma))\mathcal{F}_k(f)\\
&=\sigma^{m}\varphi(\Cdot/\sigma)\mathcal{F}_k(\Delta_{1/(2\sigma)}^{m} f)\\&
=\sigma^{m}\mathcal{F}_k(\Delta_{1/(2\sigma)}^{m}f\Ast{k}\mathcal{F}_k(\varphi(\Cdot/\sigma)))
\end{align*}
where
\[
\varphi(y)=\frac{|y|^{m}\eta(y)}{j_{\lambda_{k},m}(|y|/2)}.
\]
Since $j_{\lambda_{k},m}(|y|/2)/|y|^{m}\in C^{\infty}(\mathbb{R}^d)$, we observe that
$\varphi\in \mathcal{S}(\mathbb{R}^d)$ and $\mathcal{F}_k(\varphi)\in L^{1}(\mathbb{R}^d,
d\mu_k)$. Then estimate \eqref{normapproximant2} follows from condition $f\in L^{p}(\mathbb{R}^d,d\mu_k)$,
Lemma~\ref{lem3.4}, Proposition~\ref{prop2.4}, and
$\|\mathcal{F}_k(\varphi(\Cdot/\sigma))\|_{1,d\mu_k}=\|\mathcal{F}_k(\varphi)\|_{1,d\mu_k}$.
Lemma~\ref{lem4.7} is proved.
\end{proof}

\subsection{Proofs of Theorem \ref{thm4.1} and \ref{thm4.2} }

\begin{proof}[Proof of Theorem \ref{thm4.1}]
Using Lemma~\ref{lem4.6}, we obtain
\begin{align*}
E_\sigma (f)_{p,d\mu_k}&\leq \|f-P_{\sigma/2}(f)\|_{p,d\mu_{k}}\\&\lesssim
\sigma^{-r}\|\Delta_{1/\sigma}^m((-\Delta_k)^{r/2}f)\|_{p,d\mu_k}\\&\lesssim
\sigma^{-r}\,\omega_m\Bigl(\sigma^{-1}, (-\Delta_k)^{r/2}f\Bigr)_{p,d\mu_k}.
\end{align*}
Theorem~\ref{thm4.1} is proved.
\end{proof}

\begin{proof}[Proof of Theorem \ref{thm4.2}]
In connection with \eqref{ineqmodulus} and Lemma~\ref{lem4.4}, observe that, for
$f\in L^{p}(\mathbb{R}^d,d\mu_k)$ and $g\in W_{p, k}^{r}$,
\begin{align*}
\|\Delta_{\delta}^{r}f\|_{p,d\mu_k}&\leq\omega_r(\delta,
f)_{p,d\mu_k} \leq \omega_r(\delta, f-g)_{p,d\mu_k}
+\omega_r(\delta, g)_{p,d\mu_k}\\
&\lesssim (\|f-g\|_{p,d\mu_k}+\delta^{r}\|(-\Delta_k)^{r/2}g\|_{p,d\mu_k}).
\end{align*}
Then
\[
\|\Delta_{\delta}^{r}f\|_{p,d\mu_k}\leq \omega_r(\delta, f)_{p,d\mu_k} \lesssim K_{r}(\delta, f)_{p,d\mu_k}.
\]
Corollary~\ref{cor5.3} below implies $P_\sigma(f)\in W_{p, k}^{r}$, and
\begin{equation}\label{prof1}
K_{r}(\delta, f)_{p,d\mu_k} \leq \|f-P_\sigma (f)\|_{p,d\mu_k}+
\delta^{r}\|(-\Delta_k)^{r/2}P_\sigma(f) \|_{p,d\mu_k}.
\end{equation}
In light of Lemma \ref{lem4.6},
\begin{equation}\label{prof2}
\|f-P_{\sigma}(f)\|_{p,d\mu_{k}} \lesssim \|\Delta_{1/(2\sigma)}^rf\|_{p,d\mu_k}.
\end{equation}
Further, Lemma \ref{lem4.7} yields
\begin{equation}\label{prof3}
\|((-\Delta_k)^{r/2}P_\sigma(f)\|_{p,d\mu_k} \lesssim \sigma^{r}
\|\Delta_{1/(2\sigma)}^{r} f\|_{p,d\mu_k}.
\end{equation}
Setting $\sigma=1/(2\delta)$, from \eqref{prof1}--\eqref{prof3} we arrive at
\[
K_{r}(\delta, f)_{p,d\mu_k}\lesssim \|\Delta_{\delta}^{r}
f\|_{p,d\mu_k}\lesssim \omega_r(\delta, f)_{p,d\mu_k}.
\]
Theorem~\ref{thm4.2} is proved.
\end{proof}

\begin{remark}\label{rem4.3}
Properties \eqref{K-property1} and \eqref{K-property2} of the $K$-functional and the
equivalence \eqref{K-equiv} imply the following properties of the fractional modulus of
smoothness:

\medbreak
\textup{(1)} $\lim_{\delta\to 0+0}\omega_m(\delta,
f)_{p,d\mu_k}=0;$

\medbreak
\textup{(2)} $\omega_m(\lambda\delta, f)_{p,d\mu_k}\lesssim\max\{1,
\lambda^{2m}\}\omega_m(\delta, f)_{p,d\mu_k}.$
\end{remark}

\section{Some inequalities for entire functions}

In this section, we study weighted and fractional analogues of the inequalities for entire functions.
In particular, we obtain  in fractional case  Bernstein's inequality  (Corollary
\ref{cor5.3}), Nikolskii--Stechkin's inequality  (Corollary
\ref{cor5.5}), and Boas-type inequality (Corollary \ref{cor5.6}).

\begin{lemma}\label{lem5.1}
If $\sigma,t>0$, $r,m\geq 0$, $1 \leq p\leq \infty$, $f\in B_{p, k}^\sigma$, then
\[
\Delta_{t}^{m}(-\Delta_k)^{r/2}f,\,\, (-\Delta_k)^{r/2}\Delta_{t}^{m}f\in L^{p}(\mathbb{R}^d,d\mu_k),
\]
and
\[
\Delta_{t}^{m}(-\Delta_k)^{r/2}f=(-\Delta_k)^{r/2}\Delta_{t}^{m}f.
\]
\end{lemma}

\begin{proof} By Proposition~\ref{prop3.1} we can assume $f\in\Phi_k'$ and $\mathrm{supp}\,
\mathcal{F}_k(f)\subset B_{\sigma}$. Applying \eqref{distributiontransform}, \eqref{transformconvolution2},
and \eqref{transformdifference1} we obtain
\begin{align*}
\mathcal{F}_k(\Delta_{t}^{m}(-\Delta_k)^{r/2}f)&=\mathcal{F}_k((-\Delta_k)^{r/2}\Delta_{t}^{m}f)
=|\Cdot|^rj_{\lambda_{k},m}(t|\Cdot|)\mathcal{F}_k(f)\\&
=|\Cdot|^rj_{\lambda_{k},m}(t|\Cdot|)\eta(\Cdot/\sigma)\mathcal{F}_k(f)=g\mathcal{F}_k(f)=\mathcal{F}_k(f\Ast{k}\mathcal{F}_k(g)),
\end{align*}
where
\[
g(y)=t^m|y|^{r+m}\frac{j_{\lambda_{k},m}(t|y|)}{(t|y|)^m}\eta\Bigl(\frac{|y|}{\sigma}\Bigr),\quad \mathcal{F}_k(g)\in N_k.
\]
Hence,
\[
(f\Ast{k}\mathcal{F}_k(g))=\Delta_{t}^{m}(-\Delta_k)^{r/2}f=(-\Delta_k)^{r/2}\Delta_{t}^{m}f.
\]
Since $\frac{j_{\lambda_{k},m}(t|y|)}{(t|y|)^m}\in C^{\infty}(\mathbb{R}^d)$ and
$\frac{j_{\lambda_{k},m}(t|y|)}{(t|y|)^m}\eta(|y|/\sigma)\in\mathcal{S}(\mathbb{R}^d)$, then by
Lemma~\ref{lem3.3}, $\mathcal{F}_k(g)\in L^1(\mathbb{R}^d,d\mu_k)$. If $f\in L^p(\mathbb{R}^d,d\mu_k)$,
then the statements of Lemma~\ref{lem5.1} follow from Lemma~\ref{lem3.4} and Proposition~\ref{prop2.4}.
\end{proof}

Quantitative estimates of the norms of entire functions will be given in the following theorem.

\begin{theorem}\label{thm5.2}
If $\sigma>0$, $r_1,r_2\geq 0$, $m_1,m_2\geq 0$, $\rho=r_1+m_1-r_2-m_2\geq 0$, $1 \leq p\leq \infty$,
$0<\delta\leq t\leq 1/(2\sigma)$, and $f\in B_{p, k}^\sigma$, then
\begin{equation}\label{eq5.1}
\delta^{-m_1}\|\Delta_{\delta}^{m_1}(-\Delta_k)^{r_1/2}f\|_{p,d\mu_k}\lesssim
\sigma^{\rho}t^{-m_2}\|\Delta_{t}^{m_2}(-\Delta_k)^{r_2/2}f\|_{p,d\mu_k},
\end{equation}
where constants do not depend on $\sigma, \delta,t$, and $f$.
\end{theorem}

\begin{proof}
As in the proof of Lemma~\ref{lem5.1}, we have
$f\in\Phi_k'$, $\mathrm{supp}\,\mathcal{F}_k(f)\subset B_{\sigma}$, and
\begin{align*}
\mathcal{F}_k(\Delta_{\delta}^{m_1}(-\Delta_k)^{r_1/2}f)=|\Cdot|^{r_1-r_2}\frac{j_{\lambda_{k},m_1}(\delta|\Cdot|)}
{j_{\lambda_{k},m_2}(t|\Cdot|)}\eta\Bigl(\frac{|\Cdot|}{\sigma}\Bigr)|\Cdot|^{r_2}j_{\lambda_{k},m_2}(t|\Cdot|)\mathcal{F}_k(f).\\&
\end{align*}
Since for $0<t\leq 1/(2\sigma)$
\[
\eta(y/\sigma)=\eta(y/\sigma)\eta(ty),
\]
we obtain
\[
\delta^{-m_1}\mathcal{F}_k(\Delta_{\delta}^{m_1}(-\Delta_k)^{r_1/2}f)=\sigma^{\rho}t^{-m_2}\chi(\Cdot/\sigma)\varphi_{\theta}(t\Cdot)
\mathcal{F}_k(\Delta_{\delta}^{m_2}(-\Delta_k)^{r_2/2}f)
\]
\[
=\sigma^{\rho}t^{-m_2}\mathcal{F}_k(\Delta_{\delta}^{m_2}
(-\Delta_k)^{r_2/2}f\Ast{k}(\mathcal{F}_k(\chi(\Cdot/\sigma))\Ast{k}\mathcal{F}_k(\varphi_{\theta}(t\Cdot)))),
\]
where
\[
\theta=\frac{\delta}{t}\in (0, 1],\quad \chi(y)=|y|^{\rho}\eta(y),\quad
\psi_1(y)=\frac{j_{\lambda_{k},m_1}(|y|)}{|y|^{m_1}},
\]
\[
\psi_2(y)=\frac{j_{\lambda_{k},m_2}(|y|)}{|y|^{m_2}},\quad\varphi_{\theta}(y)=\frac{\psi_1(\theta y)}{\psi_2(y)}\eta(y).
\]
Hence,
\[
\delta^{-m_1}\Delta_{\delta}^{m_1}(-\Delta_k)^{r_1/2}f=\sigma^{\rho}t^{-m_2}(\Delta_{\delta}^{m_2}
(-\Delta_k)^{r_2/2}f\Ast{k}(\mathcal{F}_k(\chi(\Cdot/\sigma))\Ast{k}\mathcal{F}_k(\varphi_{\theta}(t\Cdot)))).
\]
By Lemma~\ref{lem3.3}
\[
\mathcal{F}_k(\chi(\Cdot/\sigma))\in N_k\cap L^1(\mathbb{R}^d,d\mu_k),\quad
\|\mathcal{F}_k(\chi(\Cdot/\sigma))\|_{1,d\mu_k}=\|\mathcal{F}_k(\chi)\|_{1,d\mu_k},
\]
\[
\psi_1(y),\,\psi_2(y)\in C^{\infty}(\mathbb{R}^d),\,\,\psi_2(y)>0,\quad \varphi_{\theta}(y)\in\mathcal{S}(\mathbb{R}^d)\times C^{\infty}([0,1]),
\]
\[
\mathcal{F}_k(\varphi_{\theta}(t\Cdot))\in N_k\cap L^1(\mathbb{R}^d,d\mu_k),\quad
\|\mathcal{F}_k(\varphi_{\theta}(t\Cdot)))\|_{1,d\mu_k}=\|\mathcal{F}_k(\varphi_{\theta})\|_{1,d\mu_k}.
\]
Applying for $f\in L^p(\mathbb{R}^d,d\mu_k)$, Lemma~\ref{lem3.4} and Proposition~\ref{prop2.4} two times, we obtain
\begin{align*}
&\delta^{-m_1}\|\Delta_{\delta}^{m_1}(-\Delta_k)^{r_1/2}f\|_{p,d\mu_k}\\
&\qquad \leq\sigma^{\rho}t^{-m_2}\|\mathcal{F}_k(\chi)\|_{1,d\mu_k}
\max_{0\leq\theta\leq 1}\|\mathcal{F}_k(\varphi_{\theta})\|_{1,d\mu_k}\|\Delta_{t}^{m_2}(-\Delta_k)^{r_2/2}f\|_{p,d\mu_k}\\
&\qquad \lesssim\sigma^{\rho}t^{-m_2}\|\Delta_{t}^{m_2}(-\Delta_k)^{r_2/2}f\|_{p,d\mu_k}
\end{align*}
provided that the function $n(\theta)=\|\mathcal{F}_k(\varphi_{\theta})\|_{1,d\mu_k}$ is continuous on $[0, 1]$.
Let us prove this.

Set
 $\varphi_{\theta}(y)=\varphi_{\theta 0}(|y|)$, $r=|y|$, $\rho=|x|$. Then by \eqref{averaging}
\begin{align*}
n(\theta)&=\int_{\R^d}\Bigl|\int_{\mathbb{R}^d}\varphi_{\theta}(y)e_k(x,y)\,d\mu_k(y)\Bigr|\,d\mu_k(x)\\
&=\int_{0}^{\infty}\Bigl|\int_{0}^2\varphi_{\theta 0}(r)j_{\lambda_k}(\rho
r)\,d\nu_{\lambda_k}(r)\Bigr|\,d\nu_{\lambda_k}(\rho)\\
&=b_{\lambda_k}^2\int_{0}^{\infty}\Bigl|\int_{0}^2\varphi_{\theta 0}(r)j_{\lambda_k}(\rho
r)r^{2\lambda_k+1}\,dr\Bigr|\rho^{2\lambda_k+1}\,d\rho.
\end{align*}
The inner integral continuously depends on $\theta$. Let us show that the outer integral converges uniformly in
 $\theta\in [0,1]$. Since
\[
\frac{d}{dr}\bigl(j_{\lambda_k+1}(\rho r)r^{2\lambda_k+2}\bigr)=(2\lambda_k+2)j_{\lambda_k}(\rho r)r^{2\lambda_k+1},
\]
integrating by parts implies
\begin{align*}
&\int_{0}^2\varphi_{\theta 0}(r)j_{\lambda_k}(\rho
r)r^{2\lambda_k+1}\,dr=
\frac{1}{2\lambda_k+2}\int_{0}^2\varphi_{\theta 0}(r)\,d\bigl(j_{\lambda_k+1}(\rho
r)r^{2\lambda_k+2}\bigr)\\
&\qquad =-\frac{1}{2\lambda_k+2}\int_{0}^2\frac{(\varphi_{\theta 0}(r))'}{r}j_{\lambda_k+1}(\rho
r)r^{2\lambda_k+3}\,dr=\dots\\
&\qquad =(-1)^s\Bigl(\prod_{j=1}^s(2\lambda_k+2s)\Bigr)^{-1}
\int_{0}^2
\varphi_{\theta 0}^{[s]}(r)j_{\lambda_k+s}(\rho r)r^{2\lambda_k+2s+1}\,dr,
\end{align*}
where
\[
\varphi_{\theta 0}^{[s]}(r):=\frac{\frac{d}{dr}\varphi_{\theta 0}^{[s-1]}(r)}{r}\in C^{\infty}(\mathbb{R}_+\times [0,1]),
\]
since $\varphi_{\theta 0}(r)$ is even in $r$ and $\varphi_{\theta 0}\in C^{\infty}(\mathbb{R}_+\times [0,1])$.
This and \eqref{besselestimate2} give
\[
\Bigl|\int_{0}^2\varphi_{\theta 0}(r)j_{\lambda_k}(\rho
r)r^{2\lambda_k+1}\,dr\Bigr|\leq \frac{c_1(\lambda_k, m_1,m_2,
s)}{(\rho+1)^{\lambda_k+s+1/2}}
\]
and, for $s>\lambda_k+3/2$,
\[
n(\theta)\leq c_2(\lambda_k, m_1,m_2, s)\int_{0}^{\infty}(1+\rho)^{-(s-\lambda_k-1/2)}\,d\rho\leq c_3(\lambda_k, m_1,m_2, s),
\]
completing the proof of continuity of $n(\theta)$. Theorem~\ref{thm5.2} is proved.
\end{proof}

We give some special cases of inequality \eqref{eq5.1}.

\begin{corollary}[Bernstein's inequality \cite{Nik75}]\label{cor5.3}
If $\sigma>0$, $r>0$, $1 \leq p\leq \infty$, $f\in B_{p, k}^\sigma$, then
\begin{equation}\label{eq5.2}
\|(-\Delta_k)^{r/2}f\|_{p,d\mu_k} \lesssim \sigma^{r}\|f\|_{p,d\mu_k}.
\end{equation}
\end{corollary}


\medskip

The next result follows from Lemma \ref{lem4.4} and Corollary~\ref{cor5.3}.

\begin{corollary}\label{cor5.4}
If $\sigma,\,\delta>0$, $m>0$, $1 \leq p\leq \infty$, $f\in B_{p, k}^\sigma$, then
\[
\omega_m(\delta, f)_{p,d\mu_k}\lesssim (\sigma\delta)^{m}\|f\|_{p,d\mu_k},
\]
where constants do not depend on $\sigma, \delta,$ and $f$.
\end{corollary}

\medskip
\begin{corollary}[Nikolskii--Stechkin's inequality \cite{Nik48,Ste48}]\label{cor5.5}
If $\sigma>0$, $m>0$, $1 \leq p\leq \infty$, $0<t\leq 1/(2\sigma)$, $f\in B_{p, k}^\sigma$, then
\begin{equation}\label{eq5.3}
\|(-\Delta_k)^{m/2}f\|_{p,d\mu_k} \lesssim t^{-m}\|\Delta_{t}^{m} f\|_{p,d\mu_k}.
\end{equation}
\end{corollary}

\medskip
\begin{remark}\label{rem5.1}
By Theorem~\ref{thm4.2}, this inequality can be equivalently written as
\[
\|(-\Delta_k)^mf\|_{p,d\mu_k} \lesssim t^{-m}
K_{m}(t, f)_{p,d\mu_k}.
\]
\end{remark}

\medskip
\begin{corollary}[Boas' inequality \cite{Boa48}]\label{cor5.6}
If $\sigma>0$, $m>0$, $1 \leq p\leq \infty$, $0<\delta\leq t\leq 1/(2\sigma)$, $f\in B_{p, k}^\sigma$, then
\begin{equation}\label{eq5.4}
\delta^{-m}\|\Delta_{\delta}^{m} f\|_{p,d\mu_k} \lesssim t^{-m}\|\Delta_{t}^{m} f\|_{p,d\mu_k}.
\end{equation}
\end{corollary}

\medskip
\begin{remark}\label{rem5.2}
Using Theorem~\ref{thm4.2} and taking into account that by \eqref{K-property2}
$\delta^{-m}K_{m}(\delta, f)_{p,d\mu_k}$ is almost decreasing in $\delta$,   inequality (\ref{eq5.4}) can be equivalently written as
\begin{align*}
&\delta^{-m}\|\Delta_{\delta}^{m} f\|_{p,d\mu_k} \asymp t^{-m}\|\Delta_{t}^{m} f\|_{p,d\mu_k},
\\
&\delta^{-m}
 K_{m}(\delta, f)_{p,d\mu_k}
 \asymp t^{-m}K_{m}(t, f)_{p,d\mu_k}.
\end{align*}
\end{remark}

\section{Realization of $K$-functional and modulus of smoothness}

Let the realization of the $K$-functional $K_{r}(t, f)_{p,d\mu_k}$ be given as follows:
\[
\mathcal{R}_{r}(t, f)_{p,d\mu_k}=\inf\{\|f-g\|_{p,d\mu_k}+t^{r}\|(-\Delta_k)^{r/2}g\|_{p,d\mu_k}\colon
g\in B_{p, k}^{1/t}\}
\]
and
\[
\mathcal{R}^*_{r}(t, f)_{p,d\mu_k}=\|f-g^*\|_{p,d\mu_k}+t^{r}\|(-\Delta_k)^{r/2}g^*\|_{p,d\mu_k},
\]
where $g^*\in B_{p, k}^{1/t}$ is the best or near best approximant for $f$ in $L^{p}(\mathbb{R}^d, d\mu_k)$.

The realization of the $K$-functional was defined in \cite{DHI}, where the importance of this concept in the
theory of approximations was shown.

\begin{theorem}\label{thm6.1}
If $t> 0$, $1\leq p\leq\infty$, $r>0$, then for any $f\in
L^{p}(\mathbb{R}^{d},d\mu_{k})$
\begin{equation*}
\begin{aligned}
\mathcal{R}_{r}(t, f)_{p,d\mu_k}&\asymp
\mathcal{R}^*_{r}(t, f)_{p,d\mu_k}
\asymp
K_{r}(t, f)_{p,d\mu_k}\asymp\omega_r(t, f)_{p,d\mu_k}.
\end{aligned}
\end{equation*}
\end{theorem}
\begin{proof}By Theorem \ref{thm4.2},
\begin{equation*}
\begin{aligned}
\omega_r(t, f)_{p,d\mu_k}&\asymp
K_{r}(t, f)_{p,d\mu_k}\leq
\mathcal{R}_{r}(t, f)_{p,d\mu_k}\leq
\mathcal{R}^*_{r}(t, f)_{p,d\mu_k},
\end{aligned}
\end{equation*}
where we have used the fact that
$B_{p, k}^{1/t}\subset W_{p, k}^{r}$, which follows from Lemma~\ref{lem5.1}.

Therefore, it is enough to show that
\[
\mathcal{R}^*_{r}(t, f)_{p,d\mu_k} \le C\omega_r(t, f)_{p,d\mu_k}.
\]
Indeed, for $g^*$ being the best approximant (or near best approximant), the Jackson inequality given in
Theorem~\ref{thm4.1} implies that
\begin{equation}\label{eq6.1}
\|f-g^{\ast}\|_{p,d\mu_k}\lesssim E_{1/t}(f)_{p,d\mu_k}\lesssim
 \omega_r(t, f)_{p,d\mu_k}.
\end{equation}
Using the inequality \eqref{eq5.3} and taking into account  \eqref{eq6.1}, we have
\begin{align*}
\|(-\Delta_k)^{r/2}g^*\|_{p,d\mu_k}
&\lesssim t^{-r}\|\varDelta_{t/2}^{r} g^*\|_{p,d\mu_k}
\\ &\lesssim t^{-r}\|\varDelta_{t/2}^{r} (g^*-f)\|_{p,d\mu_k}+t^{-r}\|\varDelta_{t/2}^{r} f\|_{p,d\mu_k}
\\
&\lesssim
t^{-r}\|g^*-f\|_{p,d\mu_k}+t^{-r}\omega_r(t/2, f)_{p,d\mu_k}.
\end{align*}
Using again \eqref{eq6.1}, we arrive at
\[
\|f-g^*\|_{p,d\mu_k}+t^{r}\|(-\Delta_k)^rg^*\|_{p,d\mu_k}\lesssim \omega_r(t, f)_{p,d\mu_k},
\]
completing the proof.
\end{proof}

In classical case $(k\equiv 0)$ Theorem~\ref{thm6.1} was proved in \cite{gogat}.

The next result answers the following question (see, e.g., \cite{Iva11,Rat94,timan}): when does the relation
\begin{equation}\label{eq6.2}
\omega_m\Bigl(\frac1n, f\Bigr)_{p,d\mu_k} \asymp E_n(f)_{p,d\mu_k}
\end{equation}
hold?
\begin{theorem}\label{thm6.2}
Let $1\leq p\leq\infty$ and $m>0$. We have that (\ref{eq6.2}) is valid if and only if for some $r>0$
\begin{equation}\label{eq6.3}
\omega_m\Bigl(\frac{1}{n}, f\Bigr)_{p,d\mu_k} \asymp \omega_{m+r}\Bigl(\frac{1}{n}, f\Bigr)_{p,d\mu_k}.
\end{equation}
\end{theorem}
\begin{proof}
At first we prove the non-trivial part that \eqref{eq6.3} implies \eqref{eq6.2}.
Since, by Remark~\ref{rem4.3}, part (2), we have $\omega_m(nt, f)_{p,d\mu_k} \lesssim n^{m}
\omega_m(t, f)_{p,d\mu_k}$, relation \eqref{eq6.3} implies that
\begin{equation}\label{eq6.4}
\omega_{m+r}(nt, f)_{p,d\mu_k} \lesssim n^{m} \omega_{m+r}(t, f)_{p,d\mu_k}.
\end{equation}
This and Jackson's inequality give
\begin{align*}
&\frac{1}{n^{m+r}} \sum_{j=0}^n
(j+1)^{m+r-1} E_j(f)_{p,d\mu_k}
\\ &\qquad \lesssim
\frac{1}{n^{m+r}} \sum_{j=0}^n
(j+1)^{m+r-1}
\omega_{m+r}\Bigl(\frac{1}{j+1}, f\Bigr)_{p,d\mu_k}
\\ &\qquad \lesssim
\frac{1}{n^{m+r}} \sum_{j=0}^n
(j+1)^{m+r-1}
\Bigl(\frac{n}{j+1}\Bigr)^m\lesssim
\omega_{m+r}\Bigl(\frac1n, f\Bigr)_{p,d\mu_k}.
\end{align*}
Moreover, Theorem \ref{thm7.1} below implies
\begin{align*}
\omega_{m+r}\Bigl(\frac1{l n}, f\Bigr)_{p,d\mu_k}
&\lesssim \frac{1}{(ln)^{m+r}} \sum_{j=0}^{ln}
(j+1)^{m+r-1} E_j(f)_{p,d\mu_k}
\\ &\lesssim \frac{1}{l^{m+r}}
\omega_{m+r}\Bigl(\frac{1}{n}, f\Bigr)_{p,d\mu_k}
\\ &\qquad +\frac{1}{(ln)^{m+r}} \sum_{j=n+1}^{ln}
(j+1)^{m+1r-1} E_j(f)_{p,d\mu_k},
\end{align*}
or, in other words,
\begin{align*}
&\frac{1}{n^{m+r}} \sum_{j=n+1}^{ln}
(j+1)^{m+r-1} E_j(f)_{p,d\mu_k}
\\ &\qquad \gtrsim
Cl^{m+r}\omega_{m+r}\Bigl(\frac{1}{ln}, f\Bigr)_{p,d\mu_k}-
\omega_{m+r}\Bigl(\frac{1}{n}, f\Bigr)_{p,d\mu_k}.
\end{align*}
Using again \eqref{eq6.4}, we obtain
\[
\frac{1}{n^{m+r}} \sum_{j=n+1}^{ln}
(j+1)^{m+r-1} E_j(f)_{p,d\mu_k}\gtrsim
(Cl^{r}-1)\omega_{m+r}\Bigl(\frac1n, f\Bigr)_{p,d\mu_k}.
\]
Taking into account monotonicity of $ E_j(f)_{p,d\mu_k}$ and choosing $l$
sufficiently large, we arrive at \eqref{eq6.2}.

In order to prove that \eqref{eq6.2} implies~\eqref{eq6.3} for any $r>0$ we apply the inequality
\eqref{omega-inequality2} and Jackson's inequality~\eqref{Jackson}:
\[
\omega_{m+r}\Bigl(\frac{1}{n}, f\Bigr)_{p,d\mu_k}\leq c(r/2)\omega_{m}\Bigl(\frac{1}{n}, f\Bigr)_{p,d\mu_k}
\lesssim E_n(f)_{p,d\mu_k}\lesssim\omega_{m+r}\Bigl(\frac{1}{n}, f\Bigr)_{p,d\mu_k}.
\]
Theorem~\ref{thm6.2} is proved.
\end{proof}

\begin{remark}
If for some $r>0$ the property~\eqref{eq6.3} is true, then it is true for any $r>0$.
\end{remark}

\section{Inverse theorems of approximation theory}

\begin{theorem}\label{thm7.1}
Let $m>0$, $n \in \mathbb{N}$, $1\leq p\leq\infty$, $f \in L^{p}(\mathbb{R}^d, d\mu_k)$. We have
\begin{equation}\label{eq7.1}
K_{m}\Bigl(\frac{1}{n}, f\Bigl)_{p,d\mu_k}
\lesssim \frac{1}{n^{m}} \sum_{j=0}^n
(j+1)^{m-1} E_j(f)_{p,d\mu_k}.
\end{equation}
\end{theorem}

\begin{remark}
By Theorem \ref{thm4.2}, $K_{m}\bigl(\frac{1}{n}, f\bigl)_{p,d\mu_k}$ in this
inequality can be equivalently replaced by $\omega_m\bigl(\frac{1}{n},
f\bigr)_{p,d\mu_k}$.
\end{remark}

\begin{proof}
Let us prove \eqref{eq7.1} for $\omega_m \bigl(\frac{1}{n}, f\bigr)_{p,d\mu_k}$.
By Proposition \ref{prop3.1}, for any $\sigma>0$ there exists $f_\sigma\in B_{p,
k}^\sigma$ such that
\[
\|f-f_\sigma\|_{p,d\mu_k}=E_\sigma(f)_{p,d\mu_k},\quad
E_0(f)_{p,d\mu_k}=\|f\|_{p,d\mu_k}.
\]
For any $s\in\mathbb{Z}_+$,
\begin{align*}
\omega_m(1/n, f)_{p,d\mu_k}&\leq \omega_m(1/n, f-f_{2^{s+1}})_{p,d\mu_k}+ \omega_m(1/n,
f_{2^{s+1}})_{p,d\mu_k}\\
&\lesssim E_{2^{s+1}}(f)_{p,d\mu_k}+\omega_m(1/n, f_{2^{s+1}})_{p,d\mu_k}.
\end{align*}
Using Lemma \ref{lem4.4}, we get
\begin{align*}
&\omega_m(1/n, f_{2^{s+1}})_{p,d\mu_k}\lesssim {n^{-m}}\|(-\Delta_k)^{m/2}
f_{2^{s+1}}\|_{p,d\mu_k}\\
&\qquad \lesssim \frac{1}{n^{m}} \Bigl(\|(-\Delta_k)^{m/2}f_1\|_{p,d\mu_k}+\sum_{j=0}^s \|(-\Delta_k)^{m/2}f_{2^{j+1}}-
(-\Delta_k)^{m/2}f_{2^j}\|_{p,d\mu_k}\Bigr).
\end{align*}
Then Bernstein inequality \eqref{eq5.2} implies that
\begin{align*}
\|(-\Delta_k)^{m/2}f_{2^{j+1}}-(-\Delta_k)^{m/2}f_{2^j}\|_{p,d\mu_k} &\lesssim
2^{m(j+1)}\|f_{2^{j+1}}-f_{2^j}\|_{p,d\mu_k}\\& \lesssim
2^{m(j+1)}E_{2^j}(f)_{p,d\mu_k},
\\
\|(-\Delta_k)^{m/2}f_1\|_{p,d\mu_k}&\lesssim E_0(f)_{p,d\mu_k}.
\end{align*}
Thus,
\[
\omega_m(1/n, f_{2^{s+1}})_{p,d\mu_k}\lesssim
\frac{1}{n^{m}}\Bigl(E_0(f)_{p,d\mu_k}+ \sum_{j=0}^s2^{m(j+1)}E_{2^j}(f)_{p,d\mu_k}\Bigr).
\]
Taking into account that
\begin{equation}\label{eq7.2}
\sum_{l=2^{j-1}+1}^{2^j} l^{m-1}E_l(f)_{p,d\mu_k} \geq
2^{m(j-1)}E_{2^j}(f)_{p,d\mu_k},
\end{equation}
we have
\begin{align*}
&\omega_m(1/n, f_{2^{s+1}})_{p,d\mu_k}\lesssim
\frac{1}{n^{m}}\Bigl(E_0(f)_{p,d\mu_k}+2^{m} E_1(f)_{p,d\mu_k}\\
&\qquad +\sum_{j=1}^s
2^{2m} \sum_{l=2^{j-1}+1}^{2^j} l^{m-1} E_l(f)_{p,d\mu_k}\Bigr)\lesssim \frac{1}{n^{m}}\sum_{j=0}^{2^s}(j+1)^{m-1}E_j(f)_{p,d\mu_k}.
\end{align*}
Choosing $s$ such that $2^s\leq n<2^{s+1}$ implies \eqref{eq7.1}. Theorem~\ref{thm7.1} is proved.
\end{proof}

Theorem \ref{thm7.1} and Jackson's inequality imply the following Marchaud inequality.
\begin{corollary}\label{cor9.2}
Let $\delta, m>0$, $1\leq p\leq\infty$, $f \in L^{p}(\mathbb{R}^d, d\mu_k)$. We have
\[
K_{m}(\delta, f)_{p,d\mu_k}
\lesssim
\delta^{m} \Bigl(
\|f\|_{p,d\mu_k}+
\int_\delta^1
{t^{-m} }
 { K_{m+1}(t, f)_{p,d\mu_k}}\,\frac{dt}{t}
 \Bigr).
\]
\end{corollary}

\begin{theorem}\label{thm9.3}
Let $1\leq p\leq\infty$, $f \in L^{p}(\R^d, d\mu_k)$ and $r>0$ be
such that $\sum _{j=1}^{\infty} j^{r-1}E_j(f)_{p,d\mu_k} <\infty.$ Then $f \in
W^{r}_{p, k}$ and, for any $m>0$, $n\in\mathbb{N}$, we have
\begin{equation}\label{eq7.3}
K_{m}\Bigl(\frac{1}{n}, (-\Delta_k)^{r/2}f\Bigl)_{p,d\mu_k}
\lesssim\frac{1}{n^{r}} \sum_{j=0}^n
(j+1)^{m+r-1}E_j(f)_{p,d\mu_k} +\sum_{j=n+1}^\infty j^{r-1}
E_j(f)_{p,d\mu_k}.
\end{equation}
\end{theorem}
\begin{remark}\label{rem9.2}
We can replace $K_{m}\bigl(\frac{1}{n}, (-\Delta_k)^{r/2}f\bigl)_{p,d\mu_k}$ by
the modulus $\omega_m\bigl(\frac{1}{n}, (-\Delta_k)^{r/2}f\bigr)_{p,d\mu_k}$.
\end{remark}

\begin{proof}Let us prove \eqref{eq7.3} for $\omega_m
\bigl(\frac{1}{n}, (-\Delta_k)^{r/2}f\bigr)_{p,d\mu_k}$. Consider
\begin{equation}\label{eq7.4}
(-\Delta_k)^{r/2}f_1+\sum_{j=0}^{\infty}\Bigl((-\Delta_k)^{r/2}f_{2^{j+1}}-(-\Delta_k)^{r/2}f_{2^{j}}\Bigr).
\end{equation}
By Bernstein's inequality \eqref{eq5.2},
\[
\|(-\Delta_k)^{r/2}f_{2^{j+1}}-(-\Delta_k)^{r/2}f_{2^j}\|_{p,d\mu_k}\lesssim
2^{(j+1)r}E_{2^j}(f)_{p,d\mu_k}\lesssim\sum_{l=2^{j-1}+1}^{2^j}l^{r-1}E_l(f)_{p,d\mu_k}.
\]
Therefore, series \eqref{eq7.4} converges to a function
$g\in L^{p}(\mathbb{R}^d, d\mu_k)$. Let us show that $g=(-\Delta_k)^{r/2}f$, i.e., $f\in W^{r}_{p, k}$. Set
\[
S_N=(-\Delta_k)^{r/2}f_1+\sum_{j=0}^N \Bigl((-\Delta_k)^{r/2}f_{2^{j+1}}-(-\Delta_k)^{r/2}f_{2^j}\Bigr).
\]
Then
\begin{align*}
\<\mathcal{F}_k(g), \psi\> &= \<g, \mathcal{F}_k(\psi)\> =\lim_{N\to\infty} \<S_N,
\mathcal{F}_k(\psi)\>=\\
&= \lim_{N\to\infty}\<\mathcal{F}_k(S_N), \psi\>=\lim_{N\to\infty}\<|y|^{r}\mathcal{F}_k(f_{2^{N+1}}), \psi\>
= \<|y|^{r} \mathcal{F}_k(f), \psi\>,
\end{align*}
where $\psi\in \Psi_k$.
Hence, $\mathcal{F}_k(g)(y) =|y|^{r}\mathcal{F}_k(f)(y)$ and $g=(-\Delta_k)^{r/2}f$
from $g\in L^{p}(\mathbb{R}^d, d\mu_k)$.

To obtain \eqref{eq7.3}, we write
\[
\omega_m(1/n, (-\Delta_k)^{r/2}f)_{p,d\mu_k}\leq\omega_m(1/n, (-\Delta_k)^{r/2}f-S_N)_{p,d\mu_k}+\omega_m(1/n, S_N)_{p,d\mu_k}.
\]
The first term is estimated as follows
\begin{align*}
\omega_m(1/n, (-\Delta_k)^{r/2}f-S_N)_{p,d\mu_k} &\lesssim\|(-\Delta_k)^{r/2}
f-S_N\|_{p,d\mu_k}\\
& \lesssim \sum_{j=N+1}^{\infty}2^{r(j+1)}E_{2^j}(f)_{p,d\mu_k}
\lesssim\sum_{l=2^N+1}^{\infty} l^{r-1}E_l(f)_{p,d\mu_k}.
\end{align*}
Moreover,
by Corollary \ref{cor5.4},
\begin{align*}
\omega_m(1/n, S_N)_{p,d\mu_k} &\leq \omega_m(1/n, (-\Delta_k)^{r/2}
f_1)_{p,d\mu_k}\\
&\qquad +\sum_{j=0}^{N}\omega_m(1/n, (-\Delta_k)^{r/2}
f_{2^{j+1}}-(-\Delta_k)^{r/2} f_{2^j})_{p,d\mu_k}\\
& \lesssim \frac{1}{n^{r}}\Bigl(E_0(f)_{p,d\mu_k}+ \sum_{j=0}^{N} 2^{(m+r)(j+1)}E_{2^j}(f)_{p,d\mu_k}\Bigr).
\end{align*}
Using \eqref{eq7.2} and choosing $N$ such that $2^N\leq n<2^{N+1}$ complete the
proof of~\eqref{eq7.3}.
\end{proof}

\end{document}